\documentclass[12pt,amstex]{amsart}
\usepackage{stmaryrd}

\usepackage{epsfig}
\usepackage{amsmath}
\usepackage{amssymb}
\usepackage{amscd}
\usepackage{graphicx}

\usepackage{pstricks}

\usepackage{tocvsec2}

\usepackage{hyperref}

\input xy
\xyoption{all}

\newtheorem{thm}{Theorem}[section]
\newtheorem{lem}[thm]{Lemma}

\newtheorem{prop}[thm]{Proposition}

\newtheorem{cor}[thm]{Corollary}
\theoremstyle{definition}
\newtheorem{de}[thm]{Definition}

\theoremstyle{remark}
\newtheorem{rem}[thm]{Remark}

\numberwithin{equation}{section}

\def \T {\mathbb{T}}
\def \N {\mathbb{N}}

\def \C {\mathbb{C}}
\def \Z {\mathbb{Z}}
\def \R {\mathbb{R}}

\def \B {\mathcal{B}}
\def \F {\mathcal{F}}

\def \H {\mathcal{H}}

\def \a {\alpha }
\def \b {\beta}
\def \ep {\epsilon}
\def \d {\delta}
\def \D {\Delta}

\def\w {\omega}

\def \exp {\text{exp}}

\begin{document}

\title{Sequences from zero entropy noncommutative toral automorphisms and Sarnak Conjecture}

\author{Wen Huang}
\author{Zhengxing Lian}
\author{Song Shao}
\author{Xiangdong Ye}

\address{Wu Wen-Tsun Key Laboratory of Mathematics, USTC, Chinese Academy of Sciences and
Department of Mathematics, University of Science and Technology of China,
Hefei, Anhui, 230026, P.R. China.}
\email{wenh@mail.ustc.edu.cn} \email{songshao@ustc.edu.cn}
\email{yexd@ustc.edu.cn}

\address{Wu Wen-Tsun Key Laboratory of Mathematics, USTC, Chinese Academy of Sciences and
Department of Mathematics, University of Science and Technology of China,
Hefei, Anhui, 230026, P.R. China, and Department of Mathematics, SUNY at Buffalo, Buffalo, NY 14260-2900, U.S.A.}
\email{lianzx@mail.ustc.edu.cn, zhengxin@buffalo.edu}


\date{May 07, 2015}
\date{May 21, 2015}
\date{Sept. 4, 2015}
\date{Sep. 8, 2015}
\date{Oct. 07, 2015}

\subjclass[2010]{Primary: 46L87, 11K31 } \keywords{noncommutative tori; M\"{o}bius function; entropy}

\thanks{Huang is partially supported by NNSF for Distinguished Young Schooler (11225105), and all
authors are supported by NNSF of China (11171320, 11371339, 11431012, 11571335).}

\begin{abstract}
In this paper we study $C^*$-algebra version of Sarnak Conjecture for noncommutative toral
automorphisms. Let $A_\Theta$ be a noncommutative torus and $\a_\Theta$ be the noncommutative
toral automorphism arising from a matrix $S\in GL(d,\Z)$. We show that if the Voiculescu-Brown
entropy of $\a_{\Theta}$ is zero, then the sequence $\{\rho(\a_{\Theta}^nu)\}_{n\in \Z}$ is a
sum of a nilsequence and a zero-density-sequence, where $u\in A_\Theta$ and $\rho$ is any state
on $A_\Theta$. Then by a result of Green and Tao \cite{GT2012-2}, this sequence is linearly disjoint from the M\"{o}bius function.
\end{abstract}

\maketitle

\markboth{Noncommutative tori and Sarnak Conjecture}{W. Huang, Z. Lian, S. Shao and X.D. Ye}




\section{Introduction}

\subsection{Sarnak Conjecture}\

\medskip

The M\"{o}bius function $\mu: \N\rightarrow \{-1,0,1\}$ is defined by
$\mu(1)=1$ and
\begin{equation}\label{M-function}
  \mu(n)=\left\{
           \begin{array}{ll}
             (-1)^k & \hbox{if $n$ is a product of $k$ distinct primes;} \\
             0 & \hbox{otherwise.}
           \end{array}
         \right.
\end{equation}

Let $(X,T)$ be a (topological) dynamical system, namely $X$ is a compact metric space and $T : X \rightarrow X$
a homeomorphism. We say a sequence $\xi$ is {\em realized} in $(X,T)$ if there is an $f\in C(X)$ and an $x\in X$
such that $\xi(n) = f(T^nx)$ for any $n\in\N$. A sequence $\xi$ is called {\em deterministic} if it is realized in a
system with zero topological entropy. Here is the conjecture by Sarnak \cite{Sar}:

\medskip

\noindent {\bf Sarnak Conjecture:}\ {\em
The M\"{o}bius function $\mu$ is linearly disjoint from any deterministic sequence $\xi$. That is,
\begin{equation}\label{Sarnak}
  \lim_{N\rightarrow \infty}\frac{1}{N}\sum_{n=1}^N\mu(n)\xi(n)=0.
\end{equation}
}

Please refer to \cite{Sar, Sar1, SU, LS} for more details and progress on this conjecture.

\subsection{Sarnak Conjecture from the viewpoint of $C^*$-algebra} \label{Sarnak C-version}\

\medskip

By Gelfand-Naimark theory, there is a direct equivalent form of Sarnak Conjecture in $C^*$-algebras.
Let $(X,T)$ be a  dynamical system. Then $T$ induces an endomorphism $\hat{T}$ of the
commutative $C^*$-algebra $C(X)$ by $\hat{T}(f)=f\circ T$. For each $x\in X$, it gives a pure state
$\rho_x$ on $C(X)$ by $\rho_x(f)=f(x)$. Hence (\ref{Sarnak}) is equivalent to
\begin{equation}\label{Sar1}
   \lim_{N\rightarrow \infty}\frac{1}{N}\sum_{n=1}^N\mu(n)\rho_x(\hat{T}^n(f))=0.
\end{equation}

The Voiculescu-Brown entropy (or contractive approximation entropy) of $\hat{T}$ is zero if and only if
the entropy of $T$ is zero \cite{KL2005}. By Gelfand-Naimark theorem, it is easy to deduce that Sarnak
conjecture is equivalent to the following statement: {\em
Let $\mathfrak{U}$ be a commutative $C^*$-algebra, and let $\a$ be an $^*$-automorphism of $\mathcal{U}$
with zero Voiculescu-Brown entropy (or zero contractive approximation entropy). Let $\rho$ be a state
on $\mathfrak{U}$. Then for each $ u \in \mathfrak{U}$
$$\lim_{N\rightarrow \infty}\frac{1}{N}\sum_{n=1}^N\mu(n)\rho(\a^nu)=0.$$
}

\medskip

Hanfeng Li \cite{Li} points out that in fact Sarnak conjecture is equivalent to the following two forms
which look much stronger. We will outline the main ideas how to prove the equivalences in Section \ref{section-pre}.

\medskip

\noindent {\bf Sarnak Conjecture (Voiculescu-Brown entropy version) \cite{Li}:}\ {\em
Let $\mathfrak{U}$ be an exact $C^*$-algebra, and let $\a$ be a $^*$-automorphism of $\mathcal{U}$ with
zero Voiculescu-Brown entropy. Then for each state $\rho$ on $\mathfrak{U}$ and each $u \in \mathfrak{U}$,
the sequence $\{\rho(\a^n u )\}_{n\in \Z}$ is linearly disjoint from M\"{o}bius function, i.e.
$$\lim_{N\rightarrow \infty}\frac{1}{N}\sum_{n=1}^N\mu(n)\rho(\a^nu)=0.$$
}

\medskip

\noindent {\bf Sarnak Conjecture (Contractive approximation entropy version) \cite{Li}:}\ {\em
Let $\mathfrak{U}$ be a $C^*$-algebra, and let $\a$ be a $^*$-automorphism of $\mathcal{U}$ with zero
contractive approximation entropy. Then for each state $\rho$ on $\mathfrak{U}$ and each $u \in \mathfrak{U}$,
the sequence $\{\rho(\a^n u )\}_{n\in \Z}$ is linearly disjoint from M\"{o}bius function, i.e.
$$\lim_{N\rightarrow \infty}\frac{1}{N}\sum_{n=1}^N\mu(n)\rho(\a^nu)=0.$$
}

Note that Voiculescu-Brown entropy version of Sarnak Conjecture was also studied in \cite{WY}.

\subsection{Main result of the paper}\

\medskip

The main aim of the paper is to study the sequence realized by noncommutative toral automorphisms with
zero Voiculescu-Brown entropy.
Let $\Theta = (\theta_{jk})_{1\le j,k\le d}$ be a real skew-symmetric $d\times d$ matrix. The
noncommutative $d$-torus $A_\Theta$ is defined as the universal $C^*$
-algebra generated by unitaries $u_1,\ldots, u_d$ subject
to the relations
$$ u_ju_k = e^{2\pi i \theta_{jk}}u_k u_j$$
for all $1 \le j, k\le d$. For any matrix $S = (s_{jk})_{1\le j,k\le d}$ in
$GL(d,\Z)$ there is an isomorphism $\a : A_{S'\Theta S}\rightarrow A_{\Theta}$ determined by
$$\a_\Theta(\hat{u}_j)=u_1^{s_{1j}}u_2^{s_{2j}}\ldots u_d^{s_{dj}}$$
for each $j = 1,\ldots ,d$, where $S'$ is the transpose of the matrix $S$, $\hat{u}_1,\ldots,\hat{u}_d$ are the generators of $A_{S'\Theta S}$. Thus when
$S'\Theta S \equiv \Theta$ (mod $M_d (\Z)$) we obtain an automorphism of
$A_\Theta$, which we refer to as a {\em noncommutative toral automorphism}.
Note that noncommutative tori are exact $C^*$-algebras. Their Voiculescu--Brown entropies are studied in \cite{KL2006}.

\medskip

Now assume that the Voiculescu-Brown entropy of $\a_\Theta$ is zero, and let $\rho$ be a state
on $A_\Theta$ and $u\in A_\Theta$. We claim that the sequence $\{\rho(\a_\Theta^n u )\}_{n\in \Z}$
is ``almost'' a nilsequence. Nilsequences are natural generalization of almost periodic sequences,
and we leave the definitions of nilsequences in Section \ref{section-pre}.
To be precise, one of main results is the following.

\medskip

\noindent {\bf Theorem A \label{main-result A}} {\em
Let $A_\Theta$ be a noncommutative torus and $\a_\Theta$ be its automorphism. If the Voiculescu-Brown
entropy of $\a_\Theta$ is zero, then for each state $\rho$ and $u\in A_\Theta$,
$\{\rho(\a_\Theta^n u )\}_{n\in \Z}$ is a sum of a nilsequence and a zero-density-sequence.
}

\medskip

Since the zero-density-sequence is always linear disjoint from M\"{o}bius function (Remark \ref{rem2.10}),
by Theorem A and Green-Tao's result in \cite{GT2012-2}, we have the following corollary:

\begin{cor}
Let $A_\Theta$ be a noncommutative torus and $\a_\Theta$ be its automorphism. If the Voiculescu-Brown
entropy of $\a_\Theta$ is zero, then for each state $\rho$ and $u\in A_\Theta$, $\{\rho(\a_\Theta^n u )\}_{n\in \Z}$
is linearly disjoint form M\"{o}bius function.
\end{cor}

\medskip

We will deduce the proof of Theorem A to the following theorem, which is also of independent interest.
Recall that an {\em integral polynomial} is a real coefficient polynomial $p(n)$ taking on integer values on the integers.

\medskip

\noindent {\bf Theorem B \label{main-result B}} {\em
Let $\mathcal{H}$ be a Hilbert space. Let $U_1,\ldots,U_d\in \B(\H)$ be unitary operators and
$p_1(n),\ldots,p_d(n)$ be integral polynomials. If the group of unitary operators,
generated by $U_1,\ldots, U_d$, is nilpotent, then for each $u,v\in \H$ the sequence
$$a_n=\langle U_1^{p_1(n)}\ldots U_d^{p_d(n)}u,v\rangle$$ is a sum of a nilsequence and a zero-density-sequence.
}

\medskip

To prove Theorem B we will need some tools developed in \cite{BL2002, L2000}. But for the case
when the group generated by $U_1,\ldots, U_d$ is abelian, we will give a direct proof.

\subsection{Organization of the paper}\

\medskip

The paper is organized as follows:
In Section \ref{section-pre} and Section \ref{section-nilsequence}, we introduce some basic conceptions and results needed in this paper.
To explain the basic ideas we study toral automorphisms in Section \ref{section-tori}. Then we study noncommutative
toral automorphisms with zero Voiculescu-Brown entropy in Section \ref{section-noncummutative tori}, and show how
we deduce Theorem A to Theorem B. For independent interest, we give a direct proof for Theorem B when the
group generated by $U_1,\ldots, U_d$ is abelian. In section \ref{section-nilcase}, we prove Theorem B.

\bigskip

\noindent {\bf Acknowledgments.}
We thank Hanfeng Li for pointing out the $C^*$-algebra versions of Sarnak
Conjecture, and helpful discussions about entropies in $C^*$-algebras. 
The second author also thanks the Department of Mathematics, SUNY at Buffalo for the hospitality when staying as a visiting student.

\section{Preliminaries}\label{section-pre}

In this section, we introduce some basic conceptions and results of $C^*$-algebra needed in this paper. And we outline the proof of the equivalence of original Sarnak conjecture
and contractive approximation entropy version \cite{Li}.

\subsection{Voiculescu-Brown entropy and Contractive approximation entropy}\

\medskip

The content of this subsection is from \cite{Brown, KL2005, KL2006}, please refer to them for more details.

\subsubsection{Voiculescu-Brown entropy}

\begin{de}\cite{Brown}
Let $A$ be a $C^*$-algebra and $\pi:A\rightarrow B(\H)$ be a faithful (possibly degenerate) $^*$-representation of $A$.
\begin{enumerate}
  \item $Pf(A)=\{\w{{:}}\, \w\subset A$ is  a  finite set $\}$.
  \item $CPA(\pi,A)=\{(\phi,\psi,B):$ where $\phi:A\rightarrow B,\psi:B\rightarrow B(\H)$ are contractive completely positive maps and $ {\rm dim}(B)<\infty\}$.
  \item $rcp(\pi,\w,\delta)=\inf\{{\rm rank}(B):(\phi,\psi,B)\in CPA(\pi,A)$ and $\|\psi\circ\phi(x)-\pi(x)\|<\delta$ for all $x\in \w\}$, where ${\rm rank}(B)$ is the dimension of a maximal abelian subalgebra of $B$.
\end{enumerate}

The notations above stand for the ``finite parts'' of $A$, the completely
positive approximations of $(\pi;A)$ and the completely positive {\it $\delta$-rank} of
$\omega$ with respect to $\pi$, respectively. The {\it $\delta$-rank} of $\omega$ is defined to be $\infty$ if
no such approximation exists.

\medskip

Now assume that we have an automorphism $\a\in {\rm Aut}(A)$, and set
\begin{equation*}
\begin{split}
ht(\pi,\a,\w,\delta)= & \lim\sup\frac{1}{n}\log rcp(\pi,\w\cup\a\w\cup\ldots\cup\a^{n-1}\w,\delta),\\
ht(\pi,\a,\w)= & \sup_{\delta>0} hc(\a,\w,\delta),\\
ht(\pi,\a)= & \sup_{{\omega}\in Pf(A)} hc(\a,\w)=\text{Voiculescu-Brown entropy of }\ \a
\end{split}
\end{equation*}
\end{de}
Then $ht(\pi,\a)$ is called the {\em Voiculescu-Brown entropy} of $\a$ ({\em VB entropy, for short}) of $\a$.

\subsubsection{Contractive approximation entropy}

Let $X$ and $Y$ be Banach spaces and $\gamma:X\rightarrow Y$ a bounded linear map. Denote by
$\mathcal{P}_f(X)$ the collection of finite subsets of $X$. For each $\Omega\in \mathcal{P}_f(X)$
and $\delta>0$ we denote by $\text{CA}(\gamma,\Omega,\delta)$ the collection of triples $(\phi,\psi,d)$
where $d$ is a positive integer and $\phi:X\rightarrow l^d_{\infty}$ and $\psi:l^d_{\infty}\rightarrow Y$
are contractive linear maps such that
$$\|\psi\circ\phi(x)-\gamma(x)\|<\delta$$
for all $x\in \Omega$. By a CA {\em embedding} of a Banach space $X$ we mean an isometric linear map $\iota$
from $X$ to a Banach space Y such that $\text{CA}(\iota,\Omega,\delta)$ is nonempty for every $\Omega\in\mathcal{P}_f(X)$
and $\delta>0$. Every Banach space admits a CA embedding; for example, the canonical map $X \rightarrow C(B_1(X^*))$ defined
via evaluation is a CA embedding, as a standard partition of unity argument
shows.

\medskip

Let $\iota:X\rightarrow Y$ be a CA embedding. For each $\Omega\in \mathcal{P}_f(X)$ and $\delta>0$ we set
$$rc(\Omega,\delta)=\inf\{d:(\phi,\psi,d)\in \text{CA}(\iota,\Omega,\delta)\}$$
This quantity is independent of the $CA$ embedding.

We denote by $\text{IA}(X)$ the collection of all isometric automorphisms of $X$. For $\a\in \text{IA}(X)$ we set
\begin{equation*}
\begin{split}
hc(\a,\Omega,\delta)= & \lim\sup\frac{1}{n}\log rc(\Omega\cup\a\Omega\cup\ldots\cup\a^{n-1}\Omega,\delta),\\
hc(\a,\Omega)= & \sup_{\delta>0} hc(\a,\Omega,\delta),\\
hc(\a)= & \sup_{\Omega\in \mathcal{P}_f(X)} hc(\a,\Omega)
\end{split}
\end{equation*}
and refer to the last quantity as the {\em contractive approximation entropy} or {\em CA entropy} of $\a$.

\subsubsection{}

Given an isometric automorphism $\a$ of a Banach space $X$, we denote
by $T_\a$ the weak$^*$ homeomorphism of the unit ball $B_1(X^*)$ of the dual of $X$
given by $T_\a(w)=w\circ \a$. The following theorem shows the relation between the $CA$ entropy of a
$C^*$ algebra $\mathfrak{U}$ with its automorphism $\hat{T}$ and the topological entropy $h_{top}(T)$ of its state space $\mathcal{S}$ with the induced map $T$.

\begin{thm}{\cite[Theorem 3.5]{KL2005}}\label{two entropy equivalent}
Let $X$ be a Banach space and $\a\in IA(X)$. Let $Z$ be a closed $T_{\a}$-invariant subset of $B_1(X^*)$
such that the natural linear map $X\rightarrow C(Z)$ is an isomorphism from $X$ to a (closed) linear
subspace of $C(Z)$. Then the following are equivalent:
\begin{enumerate}
  \item $hc(\a)>0$,
  \item $h_{top}(T_\a)>0$,
  \item $h_{top}(T_\a)=\infty$,
  \item $h_{top}(T_\a\mid_Z)>0$.
\end{enumerate}
\end{thm}
Note the $h_{top}$ denotes the usual topological entropy.


\begin{thm}{\cite[Proposition 3.1]{KL2005}}\label{teee}
Let $X$ be a compact Hausdorff space and $T : K\rightarrow  K$
a homeomorphism. Let $\a_T$ be the $^*$-automorphism of $C(X)$ given by
$\a_T ( f ) = f \circ T$ for all $f \in C(X)$. Then
$$ht(\a_T ) = hc(\a_T ) = h_{\rm top}(T ).$$
\end{thm}


\medskip

In general there is no inequality relating contractive approximation entropy and Voiculescu-Brown
entropy. But if the Voiculescu-Brown entropy is zero then contractive approximation entropy is also
zero \cite[Proposition 4.2.]{KL2005}. An interesting question \cite[Question 4.5]{KL2005} is as follows: {\em
Is there a $^*$-automorphism of an exact $C^*$-algebra for which
the Voiculescu-Brown entropy is strictly greater than the contractive approximation entropy?}


\subsection{Outline of proofs of equivalences of Sarnak Conjecture and its $C^*$-versions}\

\medskip

Since the basic ideas are similar, we only outline the proof of the equivalence of original Sarnak conjecture
and contractive approximation entropy version. The content of this subsection is pointed out to us by Hanfeng Li \cite{Li}.

By analysis in subsection \ref{Sarnak C-version} and Theorem \ref{teee} one has that
Sarnak Conjecture (Contractive approximation entropy version) implies original Sarnak conjecture. Now we show the converse. Let $\mathfrak{U}$ be a
$C^*$-algebra, and let $\a$ be a $^*$-automorphism of $\mathcal{U}$ with zero contractive approximation entropy.
Consider the state space $\mathcal{S}$ of $\mathfrak{U}$. It is well known that $\mathcal{S}$ is compact Hausdorff space under weak$^*$
topology. Define $T:\mathcal{S}\rightarrow \mathcal{S}$ as $T(\rho)=\rho\circ \a$. Then $T$ is a homeomorphism, and
$(\mathcal{S},T)$ is a topological dynamical system. By Theorem \ref{two entropy equivalent}
(i.e. \cite[Theorem 3.5]{KL2005}) the topological entropy of $(\mathcal{S},T)$ is zero, as the contractive approximation entropy of $\a$ is zero. Since for each $u\in \mathfrak{U}$, $u(\rho)=\rho(u)$
is a continuous function on $\mathcal{S}$, if Sarnak conjecture holds, it follows that
$$\lim_{N\rightarrow \infty}\frac{1}{N}\sum_{n=1}^N\mu(n)\rho(\a^nu)=\lim_{N\rightarrow \infty}\frac{1}{N}\sum_{n=1}^N\mu(n)u(T^n\rho)=0$$
for each state $\rho$ on $\mathfrak{U}$ and each $u \in \mathfrak{U}$, that is,
Sarnak Conjecture (Contractive approximation entropy version) holds.
\medskip


\subsection{Unitary representation}\

\medskip

We write $\B(\H)$ for the algebra of all bounded linear operators on
Hilbert space $\H$, and $U(\H)$ for the group of unitary operators on $\H$. Let $G$ be a topological group. A {\em unitary representation} of
$G$ on some nonzero Hilbert space $\H$ is a map $\pi: G \rightarrow  U(\H)$ which is a continuous (with respect to the strong operator topology) homomorphism, that is, a map $\pi: G \rightarrow  U(\H)$ that satisfies $\pi(xy)=\pi(x)\pi(y)$ and $\pi(x^{-1})=\pi(x)^{-1}=\pi(x)^*$, and for which $x\mapsto \pi(x)u$ is continuous from $G$ to $\H$ for any $u\in \H$.

\subsection{GNS construction}\

\medskip

The {\em Gelfand-Neumark-Segal construction} is one of the basic tools of $C^*$-algebra theory.
One may found it in most monograph on $C^*$-algebras (see for example, {\cite[Theorem 4.5.2.]{KadisonRingrose}}).

By a {\em representation} of a $C^*$-algebra $\mathfrak{U}$ on a Hilbert space $\mathcal{H}$, we mean a $^*$ homomorphism $\varphi$ from $\mathfrak{U}$ into $\B(\H)$. If, in addition, $\varphi$ is one-to-one (hence, a $^*$ isomorphism), it is described as a {\em faithful representation}. If there is a vector $x$ in $\H$ for which the linear subspace $\varphi(\mathfrak{U})x=\{\varphi(A)x: A\in \mathfrak{U}\}$ is everywhere dense in $\H$, $\varphi$ is called a {\em cyclic representation}, and $x$ is called a {\em cyclic vector} (or generating vector) for $\varphi$.

\begin{thm}\label{GNS-construction}
If $\rho$ is a state of a $C^*$-algebra $\mathfrak{U}$, there is a cyclic representation
$\pi_{\rho}$ of $\mathfrak{U}$ on a Hilbert space $\mathcal{H}_\rho$, and a unit cyclic vector
$x_\rho\in \mathcal{H}_\rho$ for $\pi_\rho$ such that
$$\rho(u)=\langle \pi_\rho(u)x_\rho,x_\rho \rangle$$
for each $u\in \mathfrak{U}$.
\end{thm}

The method used to produce a representation from a state in the proof of above theorem is called {\em GNS construction}.

\section{Nilsequences and almost nilsequences}\label{section-nilsequence}

Nilsequences were introduced in \cite{BHK05} for studying some problems in ergodic theory and additive
combinatorics that arose from the multiple ergodic averages introduced by Furstenberg in his ergodic
theoretic proof of Szemer\'{e}di's Theorem. Nilsequences are natural generalization of almost periodic
sequences, and they are systematically studied by lots of mathematicians from then on. For more details about nilsequences, please refer to \cite{BHK05, HK08, HKM, L2015} etc.

\subsection{Nilsequences}\

\medskip

\subsubsection{Almost periodic sequences}\

\medskip

Let ${\bf a}=\{a_n\}_{n\in \Z}$ be a sequence of  complex numbers. We use $\sigma {\bf a}$ denote the shifted sequence.
Thus for $k\in \Z $, $\sigma^ k {\bf a}= \{a_{n+k}\}_{n\in \Z}$.
A bounded sequence ${\bf a}=\{a_n\}_{n\in \Z}$  of  complex numbers is called {\em almost periodic sequence} if  there exit a
compact abelian group $G$, elements $x,\tau$ of $G$, and a continuous function $f:G\rightarrow \C$ such
that $a_n=f(\tau^n x)$. It is well known that a bounded sequence ${\bf a}$ of complex numbers is almost periodic if and only if the
family of translated sequences $\{\sigma^k {\bf a}: k\in \Z\}$ is relatively compact  in $l^\infty(\Z)$ under the $l^\infty$-norm.

The family of almost periodic sequences is a sub-algebra of $l^\infty(\Z)$ that is invariant under complex
conjugation, shift and uniform limits. Denote the family of almost periodic sequences by $\mathcal{AP}$.

\subsubsection{Nilsystems}\

\medskip

Let $G$ be a group. For $g, h\in G$, we write $[g, h] =
g^{-1}h^{-1}gh$ for the commutator of $g$ and $h$ and we write
$[A,B]$ for the subgroup spanned by $\{[a, b] : a \in A, b\in B\}$.
The commutator subgroups $G_j$, $j\ge 1$, are defined inductively by
setting $G_1 = G$ and $G_{j+1} = [G_j ,G]$. Let $d \ge 1$ be an
integer. We say that $G$ is {\em $d$-step nilpotent} if $G_{d+1}$ is
the trivial subgroup.

\medskip

Let $G$ be a $d$-step nilpotent Lie group and $\Gamma$ a discrete
cocompact subgroup of $G$. The compact manifold $X = G/\Gamma$ is
called a {\em $d$-step nilmanifold}. The group $G$ acts on $X$ by
left translations and we write this action as $(g, x)\mapsto gx$.
Let $\tau\in G$ and $T$ be the
transformation $x\mapsto \tau x$ of $X$. Then $(X, T)$ is
called a {\em $d$-step nilsystem}. It is easy to see that the product of two $d$-step nilsystems is still a $d$-step nilsystem.

\medskip

Let $Y$ be the closed orbit of some point $x\in X$ under $T$. Then $Y$ can
be given the structure of a nilmanifold, $Y = H/\Lambda$, where $H$
is a closed subgroup of $G$ containing $\tau$ and $\Lambda$ is a closed
cocompact subgroup of $H$. For more details please refer to \cite{L2005}, for example.

\subsubsection{Nilsequencs}

\begin{de}\cite{BHK05}
Let $d\ge 1$ be an integer and let $X = G/\Gamma$ be a
$d$-step nilmanifold. Let $\phi$ be a continuous real (or complex)
valued function on $X$ and let $\tau \in G$ and $x \in X$. The sequence
$\{\phi(\tau^n\cdot x)\}_{n\in \mathbb{Z}}$ is called a {\em  basic $d$-step nilsequence}.
A {\em $d$-step nilsequence} is a uniform limit of basic $d$-step
nilsequences.

\medskip

We denote the set of basic $d$-step nilsequences by $\mathcal{N}_d^0$, and the set of $d$-step nilsequences by $\mathcal{N}_d$. Note that $\mathcal{N}_d$ is the closure of $\mathcal{N}_d^0$ in $l^\infty(\Z)$
under the $l^\infty$-norm.
\end{de}

\begin{de}
Let $\mathcal{N}^0=\bigcup_{d\in \N} \mathcal{N}_d^0$ be the set of all basic nilsequences, and let the closure of $\mathcal{N}^0$ by $\mathcal{N}$ in $l^{\infty}(\Z)$ under the $l^\infty$-norm. The elements of $\mathcal{N}$ will be called {\em nilsequences}.
\end{de}

\begin{rem}
\begin{enumerate}
  \item Notice that a 1-step nilsequence is nothing but an almost periodic sequence.

  \item We cite the following sentences from \cite{HK08} to explain why one needs to take uniform
  limits: ``There is a technical difficulty that explains the need to take uniform limits of basic
  nilsequences, rather than just nilsequences. Namely, $\mathcal{AP}$ is closed under the uniform norm,
  while the family of basic nilsequences is not. An inverse limit of rotations on compact abelian Lie
  groups is also a rotation on a compact abelian group, but the same result does not hold for nilsystems:
  the inverse limit of a sequence  of nilmanifolds is not, in general, the homogeneous space of some
  locally compact abelian group.''
  
  \item In the definitions of nilmanifolds and nilsequences, one may require the nilpotent groups to be connected and simply connected as done in \cite{GT2010, GT2012-2}. For the reason why we can require these hypothesis, please see \cite{GT2010, L2005}.     
\end{enumerate}
\end{rem}

Since the family of nilsystems is closed under Cartesian products, one can verify the following directly.

\begin{prop}\cite{HK08}
\
For all $d\in\N$,
$\mathcal{N}_{d}^0$, $\mathcal{N}^0$, $\mathcal{N}_d$ and $\mathcal{N}$ are subalgebras of $l^\infty(\Z)$ that are invariant under complex conjugation and shift. And $\mathcal{N}_d$ and $\mathcal{N}$ are also invariant under uniform limits.
\end{prop}


\subsubsection{Characterizations of nilsequences}\

\medskip

To get better understanding of the notion of nilsequences, we cite some characterizations of nilsequences in this subsection. {\em Note that we will not use these characterizations in the paper}.

\medskip

It is well known that a sequence is almost periodic if and only if it is a uniform limit of linear combinations of exponential sequences. There is a similar characterization for $2$-step nilsequences \cite{HK08}.

\medskip
First we need define some special $2$-step nilsequences.
For $t\in \T$, a {\em quadratic exponential sequence $q(t)=\{q_n(t)\}_{n\in \Z}$} is defined by
\begin{equation*}
    q_n(t)=e\Big(\frac{n(n-1)}{2}t\Big) ,\ \forall n\in \Z,
\end{equation*}
where $e(t)=\exp (2\pi i t)$.
For $s,t\in \T$, we set
\begin{equation*}
    \kappa(s,t)=\sum_{k\in \Z} \exp (-\pi(t+k)^2) e(ks),
\end{equation*}
and define the sequence arising from Heisenberg nilsystems $\w(\a,\b)=\{\w_n(\a,\b)\}_{n\in \Z}$ as follows:
\begin{equation*}
    \w_n(\a,\b)=\kappa(n\a,n\b) e\Big(\frac{n(n-1)}{2}\a\b\Big), \ \forall n\in \Z .
\end{equation*}

We write $\mathcal{M}$ for the family of sequences of the form
\begin{equation*}
    e(s)q(t)\w(\a_1,\b_1)\cdot \ldots \cdot \w(\a_d,\b_d)
\end{equation*}
where $s,t\in \T, d\ge 0$ is an integer and $\a_1,\ldots,\a_d,\b_1,\ldots,\b_d$ are real numbers that are rationally independent modulo 1.

\begin{thm}\cite[Theorem 3.11 and Theorem 3.12]{HK08}\label{Nil2}
The space of $2$-step nilsequences $\mathcal{N}_2$ is the closed shift invariant linear space spanned by the family $\mathcal{M}$.

The linear span of the family $\mathcal{M}$ is dense in the space $\mathcal{N}_2$ under the quadratic norm, where the {\em quadratic norm} of the nilsequences $a$ is
\begin{equation*}
    \|a\|_2=\lim_{N\to \infty} \Big(\frac{1}{N}\sum_{n=0}^{N-1}|a_n|^2\Big)^{1/2}.
\end{equation*}
\end{thm}

\medskip

For $d$-step nilsequences ($d\ge 3$), it is unknown whether it has a characterization like Theorem \ref{Nil2}. But one may characterize nilsequences in another way as follows.

\medskip

For a sequence ${\bf a} = \{a_n\}_{n\in \Z}$ of  complex numbers, integers $k, j,L$, and a real number $\d >0$, if each entry in the interval $[k-L, k+L]$ is equal to the corresponding entry in the interval $[j-L,j+L]$ up
to an error of $\d$, then we write
\begin{equation*}
  {\bf a}_{[k-L,k+L]}=_\d {\bf a}_{[j-L,j+L]}.
\end{equation*}
The characterization of almost periodic sequences by
compactness can be formulated as follows:
The bounded sequence ${\bf a} = \{a_n\}_{n\in \Z}$ is almost periodic
if and only if for all $\ep>0$, there exist an integer $L\ge 1$ and a real $\d>0$ such that for any
$k,n_1,n_2 \in \Z$ whenever ${\bf a}_{[k-L,k+L]}=_\d {\bf a}_{[k+n_1-L,k+n_1+L]}$ and ${\bf a}_{[k-L,k+L]}=_\d {\bf a}_{[k+n_2-L,k+n_2+L]}$, then $|a_k-a_{k+n_1+n_2}|<\ep$. In general, Host and Kra have the following theorem:

\begin{thm}\cite[Theorem 1.1]{HKM}
Let ${\bf a} = \{a_n\}_{n\in \Z}$ be a bounded sequence of  complex numbers and $d\ge 2$ be an integer. The sequence ${\bf a}$ is a $(d-1)$-step nilsequence if and only if for all $\ep>0$, there exist an integer $L\ge 1$ and a real number $\d>0$ such that for any $(n_1,\ldots, n_d)\in \Z^d$ and
$k\in\Z$, whenever
$${\bf a}_{[k+\ep_1n_1+\ldots+\ep_dn_d-L,k+\ep_1n_1+\ldots+\ep_dn_d+L]}=_\d {\bf a}_{[k-L,k+L]}$$
for all choices of $\ep_1,\ldots, \ep_d\in \{0,1\}$ other than $\ep_1=\ldots=\ep_d=1$, then we have $|a_k-a_{k+n_1+\ldots +n_d}|<\ep$.
\end{thm}

\subsection{Almost nilsequences}

\subsubsection{Zero-density-sequence}

\begin{de}
Let $\{a_n \}_{ n\in \Z}$ be a bounded sequence of complex numbers. We say
that $\{a_n\}_{n\in \Z}$ is a {\em zero-density-sequence} if it tends to zero in  density, i.e.
$$\lim_{N\to \infty} \frac{1}{2N+1}\sum_{n=-N}^N |a_n|=0.$$
\end{de}

\begin{rem}\label{rem2.10}\
\medskip

\begin{enumerate}
  \item It is well known that a bounded sequence $\{a_n\}_{n\in \Z}$ of  complex numbers is zero-dense if and only if $$\displaystyle \lim_{N\to \infty}
  \frac{1}{2N+1} \sum_{n=-N}^N |a_n|^2 = 0$$ if and only if there is a subset $J$ of $\Z$ with density 1
  such that $\displaystyle \lim_{J\ni n\to \infty}a_n=0$.
  \item Let ${ b}=\{b_n\}_{n\in \Z}$ be a bounded sequence of complex numbers and ${ a}=\{a_n\}_{n\in \Z}$ be a zero-density-sequence.
  Then it is easy to see that ${ab}\doteq\{a_nb_n\}_{n\in \Z}$ is also a zero-density-sequence.
  \item By (2), if a sequence $\xi$ of complex numbers is linearly disjoint from M\"{o}bius function and $a$ is a
  zero-density-sequence, then $\xi+a$ is still linearly disjoint from M\"{o}bius function.
\end{enumerate}
\end{rem}

\subsubsection{Almost nilsequences}

\begin{de}
Let $\{a_n \}_{n\in \Z}$ be a bounded sequence of complex numbers and $d\in \N$. We say
that $\{a_n\}_{n\in \Z}$ is an {\em ($d$-step) almost nilsequence} if it is a sum of a ($d$-step) nilsequence and a zero-density-sequence.

\medskip

We denote the family of $d$-step almost nilsequences (resp. almost nilsequences) by $\mathcal{AN}_d$ (resp. $\mathcal{AN}$).
\end{de}


\begin{prop}\label{prop-algebra}
Let $d\in \N$. The families $\mathcal{AN}_d$ and $\mathcal{AN}$ are  subalgebras of $l^\infty(\Z)$ that are invariant under complex
conjugation, shift and uniform limits.
\end{prop}

\begin{proof}
Since the proofs for $\mathcal{AN}_d$ and $\mathcal{AN}$ are similar. We show the case of $\mathcal{AN}$.

Let $a+\xi=\{a_n+\xi_n\}_{n\in \Z}, b+\eta=\{b_n+\eta_n\}_{n\in \Z}$ be two almost nilsequences such that $a=\{a_n\}_{n\in \Z}$, $b=\{b_n\}_{n\in \Z}\in \mathcal{N}$ and $\xi=\{\xi_n\}_{n\in \Z},\eta=\{\eta_n\}_{n\in \Z}$ are zero-density-sequences. Since $a,b$ are bounded, $a\eta, b\xi$ are zero-density-sequences. Also note that $\xi+\eta$ is a zero-density-sequence. Thus $(a+\xi)(b+\eta)$ and $a+\xi+b+\eta$ are almost nilsequences. That is, $\mathcal{AN}$ is a algebra. It is easy to verify that $\mathcal{AN}$ is invariant under complex conjugation and shift. It left to show that $\mathcal{AN}$ is closed.

\medskip

Let $a^m+\xi^m\in \mathcal{AN}$ with $a^m\in \mathcal{N}$ and $\xi^m$ is zero-dense for all $m\in \N$.
Assume that the sequence $a^m+\xi^m$ converges to $b\in l^\infty(\Z)$ when $m$ goes to $+\infty$. We need show that $b\in \mathcal{AN}$.

To see this, we need the following fact: For all $a\in \mathcal{N}$ and all zero-density-sequence $\xi$, one has that $\|a+\xi \|_\infty\ge \|a\|_\infty$. (If not, there is some $k\in \Z$ such that $|a_k|>\|a+\xi \|_\infty$. Let $c=\|a+\xi \|_\infty$. A well known fact is that nilsystems are distal systems (see  \cite[Ch 4. Theorem 3]{AGH} and \cite[Theorem 2.14]{L2005}), and in particular every point is a minimal point (see, for example, \cite[Corollary on p. 160]{F}). That means every point $x$ of a nilsystem $(X,T)$ returns to its neighborhood $U$ relatively densely, which means $\{n\in \Z: T^nx\in U\}$ has a bounded gap. It follows that if $|a_k|>c$ for some $m$, then the set $\{n\in \Z: |a_n|>c\}$ has a bounded gap. In particular, $\{n\in \Z: |a_n|>c\}$ has a positive lower density. Then there is some $n$ such that $|a_n+\xi_n|>c$, which means $\|a+\xi \|_\infty> c$, a contradiction!)

By the fact above one has that $\|a^m-a^\ell\|_\infty\le \|a^m+\xi^m-(a^\ell+\xi^\ell)\|_\infty$ for all $m,\ell\in \N$, and so the sequence $\{a^m\}_{m\in \N}$ is a Cauchy sequence, and it converges to some $a\in \mathcal{N}$. Then $\{\xi^m\}_{m\in \N}$ also converges, to some $\xi$. It is obvious that $\xi$ is zero-dense. Hence $b=a+\xi\in \mathcal{AN}$. The proof is completed.
\end{proof}

\subsection{Furstenberg's Example: $\{e(p(n))\}_{n\in \Z}$
is a nilsequence}\

\medskip

Now we give the classical example by Furstenberg \cite{F} in this subsection to show $\{e(p(n))\}_{n\in \Z}$
is a nilsequence, where $e(t)=\exp (2\pi i t)$ and $p(n)$ a real coefficient polynomial.

\medskip

Let $p(n)=a_kn^k+\ldots +a_1n+a_0$. Viewing $\T^k$ as $\{\a\}\times \T^k$, we define
$T: \T^k\rightarrow \T^k$ as follows:
\begin{equation*}{\footnotesize
 T{\bf x}= \left(
     \begin{array}{cccccc}
       1 &  &  &  &  &  \\
       1 & 1 &  &  &  &  \\
        & 1 & 1 &  &  &  \\
        &  &  & \ddots & &  \\
        &  &  &  & 1 & 1 \\
     \end{array}
   \right)
   \left(
     \begin{array}{c}
       \a \\
       x_1\\
       x_2 \\
       \vdots \\
       x_k \\
     \end{array}
   \right)
   =\left(
      \begin{array}{c}
        \a \\
        x_1+\a \\
        x_2+x_1 \\
        \vdots \\
        x_k+x_{k-1} \\
      \end{array}
    \right)}
\end{equation*}
Then
\begin{equation*}{\footnotesize
 T^n{\bf x}= \left(
     \begin{array}{cccccc}
       1 &  &  &  &  &  \\
       1 & 1 &  &  &  &  \\
        & 1 & 1 &  &  &  \\
        &  &  & \ddots & &  \\
        &  &  &  & 1 & 1 \\
     \end{array}
   \right)^n
   \left(
     \begin{array}{c}
       \a \\
       x_1\\
       x_2 \\
       \vdots \\
       x_k \\
     \end{array}
   \right)
   =\left(
      \begin{array}{c}
        \a \\
        n\a+x_1\\
        \tbinom{n}{2}\a+ nx_1+x_2 \\
        \vdots \\
        \tbinom{n}{k}\a+\tbinom{n}{k-1}x_1+\ldots +nx_{k-1}+x_k \\
      \end{array}
    \right)}
\end{equation*}

Now define $\a=k! a_k$ and choose points $x_1,\ldots, x_k$ so that
\begin{equation*}
  p(n)=\tbinom{n}{k}\a+\tbinom{n}{k-1}x_1+\ldots +nx_{k-1}+x_k.
\end{equation*}
Once one shows that $\T^k$ can be viewed as a nilrotation, then one has that $\{e(p(n))\}_{n\in \Z}$ is a nilsequence, where $e(t)=\exp (2\pi i t)$.

\medskip

Now we show that $(\T^k, T)$ is isomorphic to a nilrotation.
Let $G$ be the group of $(k+2)\times (k+2)$ lower-triangular matrices with unit diagonal

$${\footnotesize
\left(
  \begin{array}{cccccc}
    1 & 0 & 0 & \ldots & 0 & 0 \\
    a_{2,1} & 1 & 0 & \ldots & 0 & 0 \\
    a_{3,1} & a_{3,2} & 1 & \ldots & 0 & 0 \\
    \vdots & \vdots & \ddots & \ldots & \vdots & \vdots \\
    a_{k+1,1} & a_{k+1,2} & a_{k+1,3} & \ldots & 1 & 0\\
    a_{k+2,1} & a_{k+2,2} & a_{k+2,3} & \ldots & a_{k+2,k+1} & 1 \\
  \end{array}
\right)}
$$

\medskip

\noindent with $a_{i,j}\in \Z$ for $2\le j< i\le k+2$ and $a_{i, 1}\in \R$ for $i\ge 2$. Let $\Gamma$ be
the subgroup of $G$ consisting of the matrices with integer entries. Then $G$ is a nilpotent Lie group and $G/\Gamma\simeq\T^{k+1}$ via

$${\footnotesize
\left(
  \begin{array}{cccccc}
    1 & 0 & 0 & \ldots & 0 & 0 \\
    a_{2,1} & 1 & 0 & \ldots & 0 & 0 \\
    a_{3,1} & a_{3,2} & 1 & \ldots & 0 & 0 \\
    \vdots & \vdots & \ddots & \ldots & \vdots & \vdots \\
    a_{k+1,1} & a_{k+1,2} & a_{k+1,3} & \ldots & 1 & 0 \\
    a_{k+2,1} & a_{k+2,2} & a_{k+2,3} & \ldots & a_{k+2,k+1} & 1 \\
  \end{array}
\right)+\Gamma \mapsto \left(
                 \begin{array}{c}
                   a_{2,1} \\
                   a_{3,1} \\
                   \vdots \\
                   a_{k+1,1} \\
                   a_{k+2,1} \\
                 \end{array}
               \right)+\Z^{k+1}
}
$$

\medskip

Let $T_g: G/\Gamma \rightarrow G/\Gamma, x\mapsto g x$, where

$${\footnotesize
g=\left(
     \begin{array}{cccccc}
       1 &  &  &  &  &  \\
       1 & 1 &  &  &  &  \\
        & 1 & 1 &  &  &  \\
        &  &  & \ddots & &  \\
        &  &  &  & 1 & 1 \\
     \end{array}
   \right)\in G.}
$$

Let
\begin{equation*}{\footnotesize
  {\bf x}=\left(
  \begin{array}{cccccc}
    1 & 0 & 0 & \ldots & 0 & 0 \\
    \a & 1 & 0 & \ldots & 0 & 0 \\
    x_1 & a_{3,2} & 1 & \ldots & 0 & 0 \\
    \vdots & \vdots & \ddots & \ldots & \vdots & \vdots \\
    x_{k-1} & a_{k+1,2} & a_{k+1,3} & \ldots & 1 & 0 \\
    x_k & a_{k+2,2} & a_{k+2,3} & \ldots & a_{k+2,k+1} & 1 \\
  \end{array}
\right)+\Gamma \sim \left(
      \begin{array}{c}
        \a \\
        x_1 \\
        x_2 \\
        \vdots \\
        x_k \\
      \end{array}
    \right)+\Z^{k+1}
.}
\end{equation*}
Then
\begin{equation*}{\footnotesize
  \begin{split}
T_g {\bf x} & =\left(
  \begin{array}{cccccc}
    1 & 0 & 0 & \ldots & 0 & 0 \\
    \a+1 & 1 & 0 & \ldots & 0 & 0 \\
    x_1+\a & a_{3,2}+1 & 1 & \ldots & 0 & 0 \\
    \vdots & \vdots & \ddots & \ldots & \vdots & \vdots \\
    x_{k-1}+x_{k-2} & a_{k+1,2}+a_{k,2} & a_{k+1,3}+a_{k,3} & \ldots & 1 & 0 \\
    x_k +x_{k-1} & a_{k+2,2}+a_{k+1,2} & a_{k+2,3}+a_{k+1,3} & \ldots & a_{k+2,k+1}+1 & 1 \\
  \end{array}
\right)+\Gamma \\ & \sim \left(
      \begin{array}{c}
        \a \\
        x_1+\a \\
        x_2+x_1 \\
        \vdots \\
        x_k+x_{k-1} \\
      \end{array}
    \right)+\Z^{k+1}\in \{\a\}\times \T^k\subseteq \T^{k+1}.
\end{split}}
\end{equation*}

\medskip

Let $X=\overline{\{T^n_g {\rm x}:n\in \Z\}}$. Then $(X,T_g)$ is isomorphic to $(\T^k, T)$.

\section{Toral automorphisms with zero entropy}\label{section-tori}

To understand the noncommutative toral automorphisms better, in this section we study the toral
automorphisms with zero entropy and try to explain some main ideas here.

For the toral automorphisms with zero entropy, it is showed that the sequence realized in this
kind of systems is a nilsequence. We will give two approachs. One is in a pure dynamical way,
and the other is from $C^*$-algebra viewpoint. The main result of the section is contained in the
study of noncommutative automorphisms. But here it is easier to understand.

\subsection{Dynamical version}\

\medskip

Let $\T^d=\R^d/\Z^d$ be a $d$-dimensional torus.
Let $T: \T^d\rightarrow \T^d$ be $Tx=Ax$, where $A$ is an automorphism of $\T^d$. The automorphism $A$
can lift to a linear automorphism $\tilde{A}$ of $\R^d$ which preserve $\Z^d$, so that
$\tilde{A}\in GL(d,\Z)$. For convenience, we still denote $\tilde{A}$ by $A$.

\medskip

Let $f\in C(\T^d)$ and $x\in \T^d$. We want to study the sequence $\{f(T^nx)\}_{n\in \Z}$. Since the
linear combinations of characters $\phi\in \widehat{\T^d}$, the dual group of $\T^d$, are dense in $C(\T^d)$,
it suffices to study the sequence $\{\phi(T^n x)\}_{n\in \Z}$.

\medskip

The entropy of a toral automorphism is given by the following theorem:

\begin{thm}\cite{Sinai}\label{entropy of torus}
Let $\{\lambda_1,\ldots,\lambda_d\}$ be the eigenvalues of $A$. Then the entropy of $T$ is:
$$h_{\rm top}(\T^d,T)=h_{\rm top}(\T^d,A)=\sum_{j=1}^d \ln^+|\lambda_j|,$$
where $\ln^+(x) =\max\{\ln x,0\}$.
\end{thm}

Now assume the entropy of $T$ is zero. By Theorem \ref{entropy of torus}, $\Pi_{j=1}^d \lambda_j$ is a
nonzero integer with  $|\lambda_j|\leq 1$ for all $1\le j\le d$, which means $|\lambda_j|=1$. According to the following Kronecker's Lemma, we have all the $\lambda_j$ are unity roots.

\begin{thm}[Kronecker's Lemma]\cite{Kronecker}\label{roots}
Let
$p(x)=\prod_{j=1}^d (x-\lambda_j)\in \Z[x]$
be a monic polynomial with integer coefficients. If $|\lambda_j |\le 1$ for all $j$ with $1\le j\le d$, then
there is some $k\ge 1 $ for which $\lambda^k_j=1$ for all $1\le j\le d$.
\end{thm}

Therefore $A\in GL(d,\Z)$ is a matrix whose eigenvalues are roots of unity. Thus there are some matrices $B$
and $P$ such that $A=P^{-1}BP$, where $B$ is an upper triangular matrix, and all the entries in the diagonal of $B$ are roots of unity.

Let $x\in \T^d$. We have
\begin{equation*}
T^n x  =A^n x = P^{-1}B^nP x.
\end{equation*}
Since $B$ is an upper triangular matrix and all the diagonal elements of $B$ are roots of unity, there is some
$m\in \N$ such that $B^{m}$ is an upper triangular matrix with all the diagonal elements being $1$. Let $C=B^{m}$,
and write $n$ as $n=tm+r$, $0\leq r\leq m-1$. Then we have
\begin{equation*}
  T^n x=P^{-1}C^tB^rP x .
\end{equation*}
Let $C=I+D$, where $D$ is strictly upper triangular matrix and hence $D^d=0$. Then for $t\ge d$, one has
\begin{equation*}
\begin{split}
C^t=(I+D)^t=\sum_{j=0}^t \tbinom{t}{j}D^j=\sum_{j=0}^{d-1} \tbinom{t}{j}D^j=\sum_{j=0}^{d-1}\frac{t!}{j!(t-j)!}D^j,
\end{split}
\end{equation*}
which means each element of $C^t$ is a polynomial in $t$. Notice that, for a fixed $r$, each element of
$A^{tm+r}=P^{-1}C^tB^rP$ is a real coefficient polynomial in $t$. That means that each coordinate of $T^{tm+r}x$ is a real coefficient polynomial in $t$, i.e.
$$T^nx=T^{tm+r}x=(f_{r1}(t),f_{r2}(t),\ldots, f_{rd}(t))',$$
where $f_{rj}(t)$ is real coefficient polynomial in $t$.

The character $\phi\in \widehat{\T^d}$ has the form $\phi: \T^d\rightarrow S^1, x\mapsto e(\langle v, x\rangle)$
for some $v\in \Z^d$, where $\langle v, x\rangle=\sum_{j=1}^dx_jv_j $ is the inner product. Hence
\begin{equation*}
  \phi(T^nx)=\phi(T^{tm + r}x)=e(\sum_{j=1}^d v_jf_{rj}(t)).
\end{equation*}
For each $r$ with $0\le r\le m-1$, $\{e(\sum_{j=1}^d v_jf_{rj}(t))\}_{t\in \Z}$ is a nilsequence since
$\sum_{j=1}^d v_jf_{rj}(t)$ is an integral polynomial in $t$.
Then it follows from the following lemma that $\{\phi(T^nx)\}_{n\in \Z}$ is a nilsequence.

\begin{lem}\label{nilsequenceTogether}
Let $\xi\in l^\infty(\Z)$ and $m\in \N$. For each $r\in \{0,1,\ldots,m-1\}$, let $\eta_r(t)=\xi(tm+r), \forall t\in \Z$.
\begin{enumerate}
   \item If the sequence $\eta_r\in l^\infty(\Z)$ is a (basic) nilsequence for each $r\in \{0,1,\ldots,m-1\}$,
   then $\xi$ is also a (basic) nilsequence.

   \item If the sequence $\eta_r\in l^\infty(\Z)$ is an almost nilsequence for each $r\in \{0,1,\ldots,m-1\}$,
   then $\xi$ is also an almost nilsequence.
\end{enumerate}
\end{lem}

\begin{proof}

First we assume that $\eta_r(t)$ is a basic $d_r$-step nilsequence for each $0\leq r\leq m-1$. Then for each $0\leq r\leq m-1$ there is some nilmanifold $X_r=G_r/\Gamma_r$, $\tau_r \in G_r$ and a continuous map $f_r$ such that $\eta_r(t)=f_r(\tau_r^t x_r)$ for some $x_r\in G_r/\Gamma_r$.
If, for some $r$, $G_r$ is not connected, after replacing $X_r$ by a larger nilmanifold and extending $f_r$ to a continuous function on this nilmafold one may assume that every $G_r$ is connected and simply-connected \cite{L2005}. In this case for each $r\in \{0,\ldots,m-1\}$, $\tau_r$ has a $m$-th root $\a_r$ (let $\alpha_r=\exp(\frac{1}{m}\log \tau_r)$ then $\a_r^m=\tau_r$). Then $\{b_{n,r}\}_n:=\{f_r(\a_r^{n-r}x_r))\}_n$ is still a basic $d_r$-step nilsequence, which satisfies $b_{tm+r,r}=\eta_r(t)$.

\medskip

Let $\Z_m=\{\overline{0},\ldots,\overline{m-1}\}$ and define $S:\Z_m\rightarrow \Z_m$ by $S(\overline{j})=\overline{j+1}$. And for each $r\in \{0,1,\ldots, m-1\}$ let $(G_r/\Gamma_r, T_r)$ be defined as $T_r(x)=\a_r x$. Now we consider the product system
$$(X,T)=(\prod_{r=0}^{m-1} G_r/\Gamma_r \times \Z_m, \prod_{r=0}^{m-1} T_r\times S).$$
Obviously $(X,T)$ is a nilsystem.

Define $g:\Z_m\rightarrow \C$ such that $g(\overline{0})=1$, $g(\overline{j})=0$ for $\overline{j}\neq \overline{0}$. And define a continuous function $f:X\rightarrow \C$ as follows
{$$f(y_0,y_1,\ldots,y_{m-1},\overline{j})=\sum_{r=0}^{m-1}g(\overline{j-r})f_r(\alpha_r^{-r}y_r).$$}
Thus $\xi(n)=f(T^n\vec{x})$, where $\vec{x}=(x_0,x_1,\ldots,x_{m-1},\overline{0})$, i.e. $\xi(n)$ is a basic nilsequence. It is easy to see that $\eta_r(t)=\xi(tm+r), \forall t\in \Z$ for each $r\in \{0,1,\ldots,m-1\}$.

\medskip

Now let $\eta_r=\{\eta_r(t)\}_{t\in \Z}$ be nilsequences for each $r\in \{0,1,\ldots,m-1\}$. Then there are some basic nilsequence $\{\eta_{r,k}\}_{k\in\Z}$ such that $\eta_{r,k}\rightarrow \eta_r, k\to\infty$ in $l^{\infty}(\Z)$. Let $\xi_k\in l^\infty(\Z)$ be defined as $\xi_k(tm+r)=\eta_{r,l}(t),\forall t\in \Z$ for each $r\in \{0,1,\ldots,m-1\}$. Then as showed above, $\xi_k$ is a basic nilsequence for each $k\in \N$. Let $\xi$ be the limit of $\{\xi_k\}_{k\in \N}$ in $l^\infty(\Z)$. Then $\xi$ is a nilsequence and $\eta_r(t)=\xi(tm+r), \forall t\in \Z$ for each $r\in \{0,1,\ldots,m-1\}$.

\medskip

The proof for the case of almost nilsequences is similar. Here we need notice the fact that if $\{\zeta(tm+r)\}_{t\in \Z}$ has zero-density for each $0\leq r\leq m-1$, then $\{\zeta(n)\}_{n\in \Z}$ itself has zero density.
\end{proof}

\subsection{$C^*$-algebra version}\

\medskip


Let $(\T^d, T)$ be a toral automorphism with zero entropy as above. Then it induces an automorphism:
$\hat{T}: C(\T^d)\rightarrow C(\T^d), u\mapsto u\circ T=u\circ A$. Let $\rho$ be a state on $C(\T^d)$
and $u\in C(\T^d)$. We want to study the sequence $\{\rho(\hat{T}^nu)\}_{n\in \Z}$. Note that, unlike
last subsection, we can not prove that the sequence $\{\rho(\hat{T}^nu)\}_{n\in \Z}$ is a nilsequence,
but we can show that it is an almost nilsequence.

\medskip

Let $$u_j(x)=e^{2\pi i x_j},$$
where $x=(x_1, \ldots, x_d)'\in \mathbb{T}^d$.
By Fourier expansion, we have that $u=\sum_{\vec{r}\in \Z^d}a_{\vec{r}}u_1^{r_1}\ldots u_d^{r_d}$.
For each $\ep>0$, there is a finite subset $F\subset \Z^d$ such that
$$\|u-\sum_{\vec{r}\in F}a_{\vec{r}}u_1^{r_1}\ldots u_d^{r_d}\|_{\infty}<\ep.$$
Together with $\|\rho\|=1$ and the fact $\mathcal{AN}$ being an algebra, one has that to show
$\{\rho(\hat{T}^nu)\}_{n\in \Z}\in \mathcal{AN}$, it suffices to show the sequence
$\{\rho(\hat{T}^n u_1^{r_1}\ldots u_d^{r_d})\}_{n\in \Z}$ is in $\mathcal{AN}$.

\medskip

It is easy to see that $\hat{T}$ is determined by the matrix $A=(a_{jk})_{1\leq j,k\leq d}\in GL(d,\Z)$:
$$\hat{T}(u_j)=u_1^{a_{j1}}\ldots u_d^{a_{jd}}.$$
Since the entropy of $(\T^d, T)$
is zero if and only if the VB entropy of $\hat{T}$ is zero \cite{KL2006}. By the analysis in the last
subsection, there is some $m\in \N$ such that for each $0\le r\le m-1$, $A^{tm+r}=\{q_{jkr}(t)\}_{j,k}$,
where each $q_{jkr}(t)$ is an integral polynomial.
Hence we have for each $v=u_1^{r_1}\ldots u_d^{r_d}$ with $r_j\in \Z$ and a fixed $r\in \{0,1,\ldots,m-1\}$,
$$\hat{T}^nv=\hat{T}^{tm+r}(v)=u_1^{p_{1r}(t)}\ldots u_d^{p_{dr}(t)},$$
where $p_{1r},\ldots,p_{dr}$ are integral polynomials.

By GNS construction, there is a Hilbert space $\mathcal{H}$ and a representation $\pi$ of $C(X)$
and a vector $w\in \H$ such that
$$\rho(\hat{T}^n v)=\langle \pi(\hat{T}^n v)w, w \rangle .$$
Let $U_j=\pi(u_j)$ for each $j$. Then one has that
$$\rho(\hat{T}^n v)=\langle \pi(\hat{T}^n v)w, w \rangle=\langle \pi(\hat{T}^{tm+r} v)w,
w \rangle=\langle U_1^{p_{1r}(t)}\ldots U_d^{p_{dr}(t)}w, w \rangle.$$

Then by Lemma \ref{nilsequenceTogether} and the following theorem, one has that the sequence
$\{\rho(\hat{T}^n v)\}_{n\in \Z}$ is an almost nilsequence.

\begin{thm}\label{commutative-case}
Let $\mathcal{H}$ be a Hilbert space. Let $U_1,\ldots,U_d\in \B(\H)$ be commutative unitary
operators and $p_1(n),\ldots,p_d(n)$ be integral polynomials. Then for each $u,v\in \H$ the
sequence $a_n=\langle U_1^{p_1(n)}\ldots U_d^{p_d(n)}u,v\rangle$ is a sum of a nilsequence
and a zero-density-sequence, i.e. an almost nilsequence.
\end{thm}

Note that Theorem \ref{commutative-case} is a special case of Theorem B. We will give
an independent proof in Section \ref{section-comm}.

\section{Noncommutative toral automorphisms with zero Voiculescu-Brown entropy}\label{section-noncummutative tori}

In this section we study noncommutative toral automorphisms with zero Voiculescu-Brown entropy. First we
review some basic facts about noncommutative tori, then we prove Theorem A assuming Theorem B.
Please refer to \cite{Rieffel, Varilly} for more details on noncommutative tori.

\subsection{Noncommutative tori}\

\medskip

There are several equivalent versions to define a noncommutative torus. For example, as mentioned in the introduction,
let $\Theta = (\theta_{jk})_{1\le j,k\le d}$ be a real skew-symmetric $d\times d$ matrix, and the
{\em noncommutative $d$-torus} $A_\Theta$ is defined as the universal $C^*$-algebra generated by unitaries $u_1,\ldots, u_d$ subject
to the relations
\begin{equation}\label{nil-condition}
  u_ju_k = e^{2\pi i \theta_{jk}}u_k u_j
\end{equation}
for all $1 \le j, k\le d$.

A noncommutative torus $A_\Theta$ can also be seen as a subset of $C^{\infty}(\T^d)$ with a product $*_1$,
where $\Theta=(\theta_{jk})_{1\leq j,k\leq d}$ is a real skew-symmetric $d\times d$ matrix as above.
Let $\mathcal{S}(\Z^d)$ be the space of complex-valued Schwartz functions on $\Z^n$. Then the product
of $A_\Theta$ is induced by the product $*_1$ of $\mathcal{S}(\Z^d)$. To be precise,
let $\F: C^{\infty}(\T^d)\rightarrow \mathcal{S}(\Z^d)$ be the Fourier transform. Then for $u,v\in A_\Theta$, $u*_1v$ is defined by
$$u*_1v=\F^{-1}(\F(u)*_1\F(v)),$$
and
$$\F(u)*_1\F(v)(\vec{p})=-4\pi^2\sum_{\vec{q}}\F(u)(\vec{q})\F(v)(\vec{p}-\vec{q})\exp(-\pi i\gamma(\vec{q},\vec{p}-\vec{q})),$$
where $\vec{p}=(p_1,\ldots,p_d)$, $\vec{q}=(q_1,\ldots,q_d)\in \Z^d$, and
$\gamma (\vec{p},\vec{q}) = \frac{1}{2}\sum_ {1\leq j,k\leq d} \theta_{jk} p_j q_k$.
Let
$$u_j(t)=e(t_j)=\exp(2\pi it_j), \ t=(t_1,\ldots,t_d)\in\T^d.$$
Then $u_1, u_2,\ldots, u_d$ generate $A_\Theta$.

\medskip

The {\em noncommutative torus algebra $\mathcal{A}_\Theta$} is defined as
$$\mathcal{A}_\Theta=\Big\{a=\sum_{\vec{r}\in \Z^d}a_{\vec{r}}u_1^{r_1}\ldots u_d^{r_d}:\
\{a_{\vec{r}}\}_{\vec{r}\in \Z^d}\in \mathcal{S}(\Z^d)\Big\}$$
where $\vec{r}=(r_1,\ldots,r_d)$. Recall that $\{a_{\vec{r}}\}_{\vec{r}\in \Z^d}\in \mathcal{S}(\Z^d)$ means that
$$p_k(a)^2=\sup_{\vec{r}\in \Z^d}(1+r_1^2+\ldots +r_d^2)^k|a_{\vec{r}}|<\infty$$
for all $k\in \N$.
And the topology of $\mathcal{A}_\Theta$ is given by the seminorms  $\{p_k\}_k$.

The restriction of the faithful normalized trace is the linear functional $\tau: \mathcal{A}_\Theta\rightarrow \C$ given by
$$\tau(a):=a_{\vec{0}}.$$
The GNS representation space $$\H_0=L^2(\mathcal{A}_\Theta,\tau)=\{\sum_{\vec{r}\in\Z^d}a_{\vec{r}}u_1^{r_1}\ldots u_d^{r_d}:\ \sum_{\vec{r}\in \Z^d}|a_{\vec{r}}|^2<\infty\}$$ of $A_\Theta$ may be described as the completion of the vector space $\mathcal{A}_\Theta$ in the Hilbert norm
$$\|a\|_{H}:=\tau(a^*a)^{\frac{1}{2}}.$$
Since $\tau$ is faithful, the GNS representation of $\mathcal{A}_\Theta$ is given by
$$\pi_0(a):b\rightarrow \underline{ab},$$
where $\underline{a}$ is the corresponding element in $\H_0$ from $a\in \mathcal{A}_\Theta$.
The closure of $\mathcal{A}_\Theta$ under operator norm is a $C^*$-algebra.

\subsection{Noncommutative toral automorphisms}\

\medskip

For any matrix $S=(s_{j,k})_{1\leq j,k\leq d}$ in $GL(d,\Z)$, there is an isomorphism
$\a_\Theta: A_{S^T\Theta S}\rightarrow A_\Theta$ determined by
\begin{equation}\label{auto}
  \a_\Theta(u_j)=u_1^{s_{1j}}u_2^{s_{2j}}\ldots u_d^{s_{dj}}
\end{equation}
for each $j=1,\ldots,d$. Thus when $S'\Theta S\equiv \Theta \ (\mod M_d(\Z))$ we obtain an automorphism of $A_{\Theta}$.
Please refer to \cite{KL2006} for more details and references.

\medskip

One of main results in \cite{KL2006} is the following theorem:

\begin{thm}\cite[Theorem 2.7.]{KL2006}\label{KerrLi2005}
Let $\a_{\Theta}$ be any non-commutative toral automorphism. Then
 $$ht(\a_{\Theta})\geq \frac{1}{2}\sum_{|\lambda_i|>1}\log |\lambda_i|$$
 where $\lambda_1,\ldots,\lambda_d$ are the eigenvalues of $S$ counted with multiplicity.
\end{thm}

\subsection{Proof of Theorem A}\

\medskip

Now we assume Theorem B holds and prove Theorem A.
Assume that $ht(\a_{\Theta})$ is zero. Let $u\in A_\Theta$ and $\rho$ be a state on $A_\Theta$.
We will study the sequence $\{\rho(\a_\Theta^n u)\}_{n\in \Z}$, and show that it is an almost
nilsequence. Now for each $\ep>0$, there is a finite set $F\subset \Z^d$
$$\|u-\sum_{\vec{r}\in F}a_{\vec{r}}u_1^{r_1}\ldots u_d^{r_d}\|_\Theta<\ep ,$$
where $\|\cdot\|_\Theta$ is the operator norm of $B(\H_0)$.

Combing this with the facts that $\|\rho\|=1$, $\a_{\Theta}$ is an isometric operator, and almost
nilsequences are closed under linear combination.
 we have that to show that the sequence $\{\rho(\a_\Theta^n u)\}_{n\in \Z}$ is an almost nilsequence,
it is suffices to show that the sequence
$\{\rho(\a_\Theta^n (u_1^{r_1}\ldots u_d^{r_d}))\}_{n\in \Z}$ for each fixed $\vec{r}\in \Z^d$ is an almost nilsequence.

\medskip

Now $\vec{r}=(r_1,\cdots,r_d)\in \Z^d$ is fixed.
Since $ht(\a_{\Theta})$ is zero, then by Theorem \ref{KerrLi2005} each eigenvalues $\lambda_i$ of $S$
satisfies $|\lambda_i|\leq 1$. By Theorem \ref{roots}, it is easy to see that
all $\lambda_i$ are roots of unity.
Let $S=P^{-1}JP$, where $J$ is the Jordan canonical form of $S$.
Then there is a positive integer $m$ such that each eigenvalue of $S^m$ is $1$.
Thus $J^m$ is an upper triangular matrix with all diagonal elements are 1.
For each $n\in \Z$, write $n$ as $n=tm+r$, $0\leq r\leq m-1$.
Hence all elements in $J^{tm+r}$ are polynomials in $t$ for each $0\leq r\leq m-1$.
Note that $S^{tm+r}=P^{-1}J^{tm+r}P$ and all elements of $S$ are integer, then all elements in $S^{tm+r}$ are integral polynomials in $t$ for each $0\leq r\leq m-1$.

For $ 0\leq r\leq m-1$, let
\begin{equation*}
 S^{n}  \left(
     \begin{array}{c}
       r_1 \\
       r_2\\
       \vdots \\
       r_d \\
     \end{array}
   \right)
   =S^{tm+r}  \left(
     \begin{array}{c}
       r_1 \\
       r_2\\
       \vdots \\
       r_d \\
     \end{array}
   \right)
   =\left(
      \begin{array}{c}
        p_{1r}(t) \\
        p_{2r}(t)\\
       \vdots \\
        p_{dr}(t) \\
      \end{array}
    \right)
\end{equation*}
for $t\in \mathbb{Z}$. Then $p_{1r}(t),p_{2r}(t),\cdots,p_{dr}(t)$ are integral polynomials.

\medskip
By (\ref{nil-condition})(\ref{auto}) inductively, we claim that there is a real coefficient polynomial $p_{0r}(t)$ in $t$ such that
\begin{equation}\label{s5}
  \a_\Theta^{tm+r}(u_1^{r_1}\ldots u_d^{r_d})=e(p_{0r}(t)) u_1^{p_{1r}(t)}\ldots u_d^{p_{dr}(t)},
\end{equation}
\begin{proof}[Proof of Claim] By (\ref{nil-condition})(\ref{auto}) inductively, there exist $q_0(n)\in \mathbb{R}$ for $n\in \mathbb{Z}$ such that
$$\a_\Theta^{n}(u_1^{r_1}\ldots u_d^{r_d})=e(q_0(n))u_1^{q_1(n)}\ldots u_d^{q_d(n)} \text{ for }n\in \mathbb{Z},$$
where $q_j(tm+r)=p_{jr}(t)$ for $0\le r \le m-1$ and $t\in \mathbb{Z}$, $j=1,2,\cdots, d$. We have
\begin{equation*}
\begin{split}
\a_\Theta^{n}(u_1^{r_1}\ldots u_d^{r_d})&
=e(q_0(n))u_1^{q_1(n)}\ldots u_d^{q_d(n)}\\
& = \a_\Theta(\a_\Theta^{n-1}(u_1^{r_1}\ldots u_d^{r_d}))\\
& =e(q_0(n-1))\a_\Theta(u_1^{q_1(n-1)}\ldots u_d^{q_d(n-1)})\\
& =e(q_0(n-1))(u_1^{s_{11}}\ldots u_d^{s_{d1}})^{q_1(n-1)}\ldots (u_1^{s_{1d}}\ldots u_d^{s_{dd}})^{q_d(n-1)}.
\end{split}
\end{equation*}
Since
$$u_1^{x_1}\ldots u_d^{x_d}u_1^{y_1}\ldots u_d^{y_d}=e\Big(\sum_{k=1}^d\sum_{j={k+1}}^dx_jy_k\theta_{jk}\Big)u_1^{x_1+y_1}\ldots u_d^{x_d+y_d},$$
one has that
\begin{align*}
&\hskip0.5cm (u_1^{s_{1j}}\ldots u_d^{s_{dj}})^{q_j(n-1)}\\
&=e\Big(\frac{\big(q_j(n-1)-1\big)q_j(n-1)}{2}\sum_{a=1}^d
\sum_{b={a+1}}^ds_{aj}s_{bj}\theta_{ba}\Big)u_1^{s_{1j}q_j(n-1)}\ldots u_d^{s_{dj}q_j(n-1)}.
\end{align*}
Then we have
\begin{equation*}
\begin{split}
q_0(n)&=q_0(n-1)+\sum_{j=1}^d\frac{(q_j(n-1)-1)q_j(n-1)}{2}\sum_{a=1}^d\sum_{b=a}^ds_{aj}
s_{bj}\theta_{ba}\\
& \quad +\sum_{c=1}^{d-1}\Big( \sum_{a=1}^d\sum_{b=a+1}^d\Big(\sum_{e=1}^c s_{be}q_e(n-1)\Big)s_{a(c+1)}q_{c+1}(n-1)\theta_{ba}\Big) \\
& = q_0(n-1)+\sum_{j=1}^d \lambda_j q_j(n-1)+\sum_{1\le j\leq k\le d} \lambda_{j,k}q_i(n-1)q_j(n-1)
\end{split}
(\text{mod } \mathbb{Z})
\end{equation*}
for some $\lambda_j,\lambda_{j,k}\in \C$.

Let $S(n)=\sum_{j=1}^d \lambda_j q_j(n-1)+\sum_{1\le j\le k\le d} \lambda_{j,k}q_i(n-1)q_j(n-1)$. Then  for each $0\leq r\leq m-1$, $S(tm+r)$ is a real coefficient polynomial on $t$. For $0\le r\le m-1$,
let $p_{0r}(0)=q_0(r)$
and define $p_{0r}(t)$ inductively such that
  $$p_{0r}(t)=p_{0r}(t-1)+\sum_{l=1}^mS(tm+r-l)$$
for $t\in \mathbb{Z}$.
   Since $\sum_{l=1}^mS(tm+r-l)$ is a real coefficient polynomial on $t$, we have that $p_{0r}(t)$ is real coefficient polynomials for each $r\in \{0,\ldots, m-1\}$.
Note that $p_{0r}(0)=q_0(r)$ and $$q_0(tm+r)=q_0((t-1)m+r)+\sum_{l=1}^mS(tm+r-l) \, (\text{mod} \mathbb{Z})$$ for $t\in \mathbb{Z}$,
one has $p_{0r}(t)=q_0(tm+r) \,\, (\text{mod} \mathbb{Z})$ for $t\in \mathbb{Z}$, $0\le r\le m-1$. Hence \eqref{s5} holds. This finishes the proof of Claim.
\end{proof}

According to Theorem \ref{GNS-construction}, there are a cyclic representation $\pi_{\rho}$ on a
Hilbert space $\H_\rho$ and a cyclic vector $w\in \H_\rho$ such that
$$\rho(\a_\Theta^{n}(u_1^{r_1}\ldots u_d^{r_d}))=\langle\pi_{\rho}(\a_\Theta^{n}(u_1^{r_1}\ldots u_d^{r_d})w,w\rangle.$$
Let $U_i=\pi_{\rho}(u_i)$ for $1\le i \le d$. Then $U_1,\cdots,U_d$ are unitary operators on  $\H_{\rho}$ and
\begin{align}\label{eq-ujk1}
U_jU_k=e^{2\pi i \theta_{jk}}U_kU_j
\end{align}
for $1\le i,j\le d$, since $u_1,\cdots,u_d$ are unitaries and
satisfy \eqref{nil-condition}.

By (\ref{s5}) for each $r\in \{0,\ldots,m-1\}$,
$\eta_r(t):=\rho\big( \a_\Theta^{tm+r}(u_1^{r_1}\ldots u_d^{r_d}) \big)$ has the form
$$\eta_r(t)=e(p_{0r}(t)) \langle U_1^{p_{1r}(t)}\ldots U_d^{p_{dr}(t)}w,w\rangle.$$
From \eqref{eq-ujk1}, we have that the group of unitary operators generated by $U_1,\ldots, U_d$ is
$2$-step nilpotent. Then by Theorem B, $\{\langle U_1^{p_{1r}(t)}\ldots U_d^{p_{dr}(t)}w,w\rangle\}_{t\in \mathbb{Z}}$
is an almost nilsequence.

Since $p_{0r}(t)$ is a real coefficient polynomial, $\{e(p_{0r}(t))\}_{t\in \mathbb{Z}}$ is a nilsequence and so the sequence $\{\eta_r(t)\}_{t\in \mathbb{Z}}$ is an almost nilsequence by Proposition \ref{prop-algebra}.
Thus by Lemma \ref{nilsequenceTogether}, we have that
$\{\rho(\a_\Theta^{n}(u_1^{r_1}\ldots u_d^{r_d}))\}_{n\in \mathbb{Z}}$ is an almost nilsequence.
This finishes the proof of Theorem A.

\section{Commutative case: Proof of Theorem \ref{commutative-case}}\label{section-comm}

Theorem \ref{commutative-case} is a special case of Theorem B. Since for commutative
case there are more tools, we give a direct proof, which is different from the one in the next section.

\medskip

We need the following Weyl's Theorem.

\begin{thm}[Weyl]\cite{Weyl}
Let $p(n)=a_kn^k+\ldots +a_1n+a_0$ be a real polynomial with at leat one coefficient among $a_1,\ldots,a_k$
irrational. Then the sequence $\{p(n)\}_{n\in \Z}$ is equidistributed modulo $1$.
\end{thm}

Note that the sequence $\{p(n)\}_{n\in \Z}$ is equidistributed modulo $1$ if and only if for any $k\neq 0$,
\begin{equation}\label{}
    \lim_{N\to \infty} \frac{1}{2N+1}\sum_{n=-N}^N e(k p(n))=\lim_{N\to \infty} \frac{1}{2N+1}\sum_{n=-N}^N \exp(2\pi i k p(n))=0.
\end{equation}

\medskip

\begin{proof}[Proof of Theorem \ref{commutative-case}]

Let $\mathcal{H}$ be a Hilbert space. Let $U_1,\ldots,U_d\in B(\H)$ be commutative unitary operators
and $p_1(n),\ldots,p_d(n)$ be integral polynomials. We want to prove that for each $u,v\in \H$ the
sequence $\langle U_1^{p_1(n)}\ldots U_d^{p_d(n)}u,v\rangle$ is an almost nilsequence.
First we may assume that $p_1(0)=\ldots=p_d(0)=0$ (if necessary we replace $g(n)$ and
$u$ by $g(n)(g(0))^{-1}$ and $g(0)u$ respectively, where $g(n)=U_1^{p_1(n)}\ldots U_d^{p_d(n)}$).
By the polarization identity (i.e. $\langle g(n)u, v \rangle=\frac 14 (\langle g(n)(u+v), u+v \rangle-\langle g(n)(u-v), u-v \rangle+i \langle g(n)(u+iv),u+i v \rangle-i\langle g(n)(u-iv), u-iv \rangle)$), it suffices to show that for each $u\in \H$ the sequence
$a_n=\langle U_1^{p_1(n)}\ldots U_d^{p_d(n)}u,u\rangle$ is a sum of a nilsequence and a zero-density-sequence.

\medskip

Let $\phi(\vec{n})=\langle U_1^{n_1}\ldots U_d^{n_d}u, u\rangle$ for $\vec{n}=(n_1,\ldots,n_d)\in \Z^d$. Since
$U_1, \ldots, U_d$ are commutative, it is easy to verify that $\phi$ is a normalized continuous
positive definite function on $\Z^d$. Then by Bochner's Theorem there is a probability measure $\mu$ on $\widehat{\Z^d}=\T^d$ such that
\begin{equation}\label{}
  \phi(\vec{n})=\int_{\T^d} e(\langle \xi, \vec{n}\rangle)d\mu(\xi),
\end{equation}
where $\xi=(\xi_1, \ldots, \xi_d)$, $\vec{n}=(n_1,\ldots,n_d)$ and $\langle \xi, \vec{n}\rangle=\sum_{j=1}^d\xi_j n_j$.

Hence we have
\begin{equation}\label{}
  a_n=\langle U_1^{p_1(n)}\ldots U_d^{p_d(n)}u,u\rangle=\int_{\T^d} e(\sum_{j=1}^d \xi_jp_j(n))d\mu(\xi)=\int_{\T^d} e(f_\xi(n))d\mu(\xi),
\end{equation}
where $f_\xi(n):=\sum_{j=1}^d\xi_j p_j(n)$ is a polynomial in $n$ for $\xi\in \mathbb{R}^d$. Notice that one has that $$f_\xi(n)+f_\gamma(n)=f_{\xi+\gamma}(n)$$
for $\xi,\gamma\in \mathbb{R}^d$.
Since all coefficients in $p_1(n),\ldots,p_d(n)$ are rational,  for $\xi\in \mathbb{R}^d$ and $\overrightarrow{k}\in \mathbb{Z}^d$, the followings hold that
\begin{itemize}
\item one of coefficients of $f_{\xi}(n)$ is irrational if and only if one of coefficients of $f_{\xi+\overrightarrow{k}}(n)$.
\item all coefficients of $f_{\xi}(n)$ are rational if and only if all coefficients of $f_{\xi+\overrightarrow{k}}(n)$ are rational.
\end{itemize}
\medskip

First we explain the idea of proof. Note that
\begin{equation*}
\begin{split}
|a_n|^2 & =\int_{\T^d}e(f_{\xi}(n))d\mu(\xi)\overline{\int_{\T^d} e(f_{\gamma}(n))d\mu(\gamma)}\\
& = \int_{\T^d}\int_{\T^d} e((f_{\xi}(n)-f_{\gamma}(n)) d\mu(\xi)d\mu(\gamma)
\\ & =\int_{\T^d}\int_{\T^d}e(f_{\xi-\gamma}(n))d\mu(\xi)d\mu(\gamma).
\end{split}
\end{equation*}
By Weyl's Theorem, if one of coefficients of $f_{\xi-\gamma}(n)$ is irrational, then
$$\lim_{N\to \infty}\frac{1}{2N+1}\sum_{n=-N}^Ne(f_{\xi-\gamma}(n))=0.$$

Let
\begin{equation*}
    W_1=\{(\xi,\gamma)\in \T^d\times \T^d: \text{one of coefficients of $f_{\xi-\gamma}(n)$ is irrational} \},
\end{equation*}
and
\begin{equation*}
    W_2=\{(\xi,\gamma)\in \T^d\times \T^d: \text{all coefficients of $f_{\xi-\gamma}(n)$ are rational} \}.
\end{equation*}
Then for each $(\xi,\gamma)\in W_1$, one has that $\lim_{N\to \infty}\frac{1}{2N+1}\sum_{n=-N}^N e(f_{\xi-\gamma}(n))=0.$
This part will correspond to the zero-density-sequence part of $a_n$. And we will show that nilsequence part of $a_n$
depends on $W_2$.

\medskip

Now we give the details. First it is easy to verify that $W_2$ is measurable.

\medskip

\noindent {\bf Case I:} {\em $\mu\times \mu (W_2)=0$.}

\medskip

In this case, by Lebesgue dominated convergence theorem we have
\begin{equation*}
\begin{split}
  \lim_{N\to \infty}\frac{1}{2N+1}\sum_{n=-N}^N|a_n|^2 & =  \lim_{N\to \infty}\frac{1}{2N+1}\sum_{n=-N}^N \int_{\T^d}\int_{\T^d}e(f_{\xi-\gamma}(n))d\mu(\xi)d\mu(\gamma)\\
  &=\lim_{N\to \infty}\frac{1}{2N+1}\sum_{n=-N}^N\int_{W_1}e(f_{\xi-\gamma}(n))d\mu\times \mu(\xi,\gamma) =0
  \end{split}
\end{equation*}

Hence $\{a_n\}_{n}$ is a zero-density-sequence.

\medskip

\noindent {\bf Case II:} {\em $\mu\times \mu (W_2)>0$.}

\medskip

Recall that $f_\xi(n)=\sum_{j=1}^d\xi_j p_j(n)$, where
$p_1(n),\ldots,p_d(n)$ are fixed integral polynomials. Since $p_1(n),\ldots,p_d(n)$ are fixed, the set
$$\F=\{f_\xi(n): \text{ all coefficients of $f_{\xi}(n)$ are rational }\}$$
is a countable set. Set $\F=\{g_k\}_{k\in I}$ with $g_k\neq g_m$ when $k\neq m\in I$, where $I$ is a countable index set.

Let
\begin{equation*}
    A=\{\xi\in \T^d: f_\xi(n)\in \F\}.
\end{equation*}
Since $A=\bigcup_{k\in I} \{\xi\in \T^d: f_\xi=g_k\}$, it is easy to check that $A$ is a measurable set.

\medskip

Define a relation: $\xi \sim \gamma$ if and only if $\xi-\gamma\in A$. Since $f_\xi(n)+f_\gamma(n)=f_{\xi+\gamma}(n)$, it is easy to verify that $``\sim''$ is an equivalence relation,  and denote the
equivalence class of $\xi$ by $[\xi]=\xi+A$. Then
$$ W_2=\bigcup_{\xi\in \T^d} [\xi]\times [\xi].$$
We claim that there is an at most countable set $\{\xi_m\}_{m}$ such that $\mu([\xi_m])>0$ for all $m$ and
\begin{equation*}
    W_2=_{\mu\times\mu} \bigcup_{m} [\xi_m]\times [\xi_m],
\end{equation*}
where $K=_\nu H$ means $\nu(K\D H)=0$. In fact, by Funini's Theorem,
\begin{equation}\label{s7}
\begin{split}
  \mu\times \mu(W_2)&= \int_{\T^d}\int_{\T^d} \chi_{W_2} (x,y)d\mu(x) d\mu(y) =\int_{\T^d}\Big(\int_{\T^d}\chi_{W_2}(x,y)d\mu(x)\Big)d\mu(y)\\
  &=\int_{\T^d}\Big(\int_{\T^d}\chi_{\bigcup_{\xi\in \T^d} [\xi]\times [\xi]}(x,y)d\mu(x)\Big)d\mu(y)\\
  &= \int_{\T^d}\Big(\int_{\T^d}\chi_{[y]\times [y]}(x,y)d\mu(x)\Big)d\mu(y)\\
  &=\int_{\T^d}\mu([y])d\mu(y)
  \end{split}
\end{equation}
Since there is only at most countable set $\{\xi_m\}_{m}$ such that $\mu([\xi_m])>0$ for all $m$, by equation $(\ref{s7})$ one has that
$$\mu\times \mu(W_2)=\sum_{m}\big(\mu([\xi_m])\big)^2=\mu\times \mu(\bigcup_{m} [\xi_m]\times [\xi_m]).$$
Hence we get what we need.

\medskip

Now let
\begin{equation}\label{}
    B=\bigcup_m [\xi_m]=\bigcup_m(A+\xi_m)\subseteq \T^d.
\end{equation}
Let $\mu_B=\mu|_{B}=\frac{\mu(\ \cdot \ \cap B)}{\mu(B)}$ and $\mu_{B^C}=\mu|_{X\setminus B}
=\frac{\mu(\ \cdot \ \cap (X\setminus B))}{\mu(X\setminus B)}$. Then we have a decomposition of $\mu$
\begin{equation*}
    \mu=p\mu_B+ (1-p)\mu_{B^C},
\end{equation*}
where $p=\mu(B)\in (0,1]$. Note that when $p=\mu(B)=1$, $\mu=\mu_B$
and $\mu_{B^C}$ will not be considered.

\medskip
Then
\begin{equation}\label{}
  a_n=\int_{\T^d} e(f_\xi(n))d\mu(\xi)=p\int_{\T^d} e(f_\xi(n))d\mu_B(\xi)+(1-p)\int_{\T^d} e(f_\xi(n))d\mu_{B^C}(\xi),
\end{equation}
Set
\begin{equation}\label{}
  b_n=p\int_{\T^d} e(f_\xi(n))d\mu_B(\xi),\quad  c_n=(1-p)\int_{\T^d} e(f_\xi(n))d\mu_{B^C}(\xi).
\end{equation}
To end the proof, we will show that $\{b_n\}_n$ is a nilsequence and $\{c_n\}_n$ is a zero-density-sequence.

By definition, $\mu_{B^C}\times \mu_{B^C}(W_2)=\mu_{B^C}\times \mu_{B^C}(\bigcup_{n} [\xi_n]\times [\xi_n])=0$.
For each $(\xi,\gamma)\in W_1$, one has that $\displaystyle \lim_{N\to \infty}\frac{1}{2N+1}\sum_{n=-N}^N e(f_{\xi-\gamma}(n))=0$.
Hence by Lebesgue dominated convergence theorem
\begin{equation*}
\begin{split}
  & \quad \quad \frac{1}{(1-p)^2}\lim_{N\to \infty}\frac{1}{2N+1}\sum_{n=-N}^N|c_n|^2 \\
  & =  \lim_{N\to \infty}\frac{1}{2N+1}\sum_{n=-N}^N \int_{\T^d}\int_{\T^d}e(f_{\xi-\gamma}(n))d\mu_{B^C}(\xi)d\mu_{B^C}(\gamma)\\
  &=\lim_{N\to \infty}\frac{1}{2N+1}\sum_{n=-N}^N\int_{W_1}e(f_{\xi-\gamma}(n))d\mu_{B^C}\times \mu_{B^C}(\xi,\gamma) =0.
  \end{split}
\end{equation*}
That is, $\{c_n\}_n$ is a zero-density-sequence.

\medskip

As to $b_n$, one has that
\begin{equation*}
\begin{split}
    b_n&=p\int_{\T^d} e(f_\xi(n))d\mu_B(\xi)=\int_{B}e(f_\xi(n))d\mu(\xi)\\
    &=  \int_{\bigcup_m (A+\xi_m)} e(f_\xi(n))d\mu(\xi) = \sum_m \int_{A+\xi_m} e(f_\xi(n))d\mu(\xi).
\end{split}
\end{equation*}
Note that $A=\bigcup_{k\in I} \{\xi\in \T^d: f_\xi=g_k\}$, and hence
\begin{equation*}
\begin{split}
    b_n&= \sum_m \int_{A+\xi_m} e(f_\xi(n))d\mu(\xi)\\
    &= \sum_m \sum_{k}\int_{\{\xi\in \T^d: f_\xi=g_k\}+\xi_m} e(f_\xi(n))d\mu(\xi)\\
    &= \sum_m \sum_{k}\mu(\{\xi\in \T^d: f_\xi=g_k\}) e\Big(g_k(n)+f_{\xi_m}(n)\Big).
\end{split}
\end{equation*}
Since for each $k,m$, $\Big\{e\big(g_k(n)+f_{\xi_m}(n)\big)\Big\}_{n\in \Z}$ is a nilsequence, $\{b_n\}_{n\in \Z}$ is also a nilsequence. The proof is completed.
\end{proof}

\section{nilpotent case: Proof of Theorem B}\label{section-nilcase}

In this section we will prove Theorem B. Our proof relies on some results developed by Bergelson and Leibman \cite{BL2002, L2000}.

\subsection{$G$-polynomials}\

\medskip

Let $G$ be a nilpotent group with unit $e_G$. For a sequence $g:\Z\rightarrow G$ in $G$, we define the
{\em derivative} of $g$ as the sequence $Dg: \Z\rightarrow G$ with $Dg(n)=g(n)^{-1}g(n+1)$. $g$ is a
{\em polynomial sequence} or {\em $G$-polynomial} if $D^dg\equiv e_G$ for some $d\in \N$.


It can be shown \cite{L1998} that for a nilgroup $G$, a sequence $g$ in $G$ is a $G$-polynomial if and
only if it is representable in the form
\begin{equation}\label{polynomial}
    g(n) = U_1^{p_1(n)}\ldots U_r^{p_r (n)},
\end{equation}
where $U_1,\ldots, U_r\in G$ and $p_1,\ldots,p_r$ are integral polynomials. If $g(0)=e_G$,
one can additionally assume that $p_1(0) =\ldots = p_r (0) = 0$.

\subsection{Decomposition of a Hilbert space}\

\medskip

Let $\H$ be a Hilbert space and let $G$ be a group of unitary operators on $\H$. Then $\H$ is the sum of
two $G$-invariant orthogonal subspaces:
\begin{align}\label{decomposition}
\H=\H^c(G)\oplus\H^{wm}(G)
\end{align}
such that $G$ has discrete spectrum on $\H^c(G)$ and it is weakly mixing on $\H^{wm}(G)$. The space $\H^c(G)$
is spanned by finite-dimensional $G$-invariant subspaces of $\H$, and consists of vectors whose orbits under the
action of $G$ are precompact:
$$\H^c(G)=\{u\in \H: Gu \ \text{is precompact in }\H \}.$$ The space $\H^{wm}(G)$ is the maximal $G$-invariant subspace
$\mathcal{M}$ of
$\H$ such that for any unitary action of $G$ on a Hilbert space $\mathcal{N}$, the space
$\mathcal{M}\otimes \mathcal{N}$ does not contain non-zero elements which are invariant
with respect to the $G$-action. If $G$ is an amenable (in particular, abelian or
nilpotent) group, then $\H^{wm}(G)$ can also be described as the maximal subspace of $\H$ such
that for any $\ep>0$ and any $u \in \H^{wm}(G)$ and $v\in \H$, the set $\{T\in G: |\langle Tu, v \rangle|>\ep \}$ has
zero density in $G$ (with respect to a F{\o}ner sequence).

\subsection{Leibman's Result}\

\medskip

Given $u\in \H$.
We say $G$ is {\em compact} on $u$ if $Gu$ is precompact in $\H$. Let $T$ be a unitary operator on $\H$.
We say that $T$ is {\em compact} on $u$ if the group $\{T^n:n\in \Z\}$ is compact on $u$. We say that $T$
is {\em weakly mixing on $u$} if for all $v\in \H$,
$$\lim_{N\to \infty}\frac{1}{2N+1} \sum_{n=-N}^N | \langle T^nu,v \rangle|=0.$$
Define
\begin{equation*}
  \H^c(T)=\{u\in \H: T \ \text{ is compact on}\ u\},
\end{equation*}
\begin{equation*}
  \H^{wm}(T)=\{u\in \H: T \ \text{ is weakly mixing on}\ u\}.
\end{equation*}
For $T\in G$, we have
\begin{equation*}
  \H=\H^c(T)\oplus\H^{wm}(T).
\end{equation*}

For a general group $G$ of unitary operators on $\H$, one can only claim $\H^c(G)\subseteq \H^c(T)$ and hence $\H^{wm}(T)\subseteq \H^{wm}(G)$ for each $T\in G$.
But if $G$ is a finitely generated nilpotent group of unitary operators, then one can say more. In fact Leibman
showed the following interesting theorem:

\begin{thm}[Leibman]\cite[Corollary 3.3]{L2000}
Let $G$ be a finitely generated nilpotent group of unitary operators on a Hilbert space $\H$, and
let $T_1,T_2,\ldots, T_r$ generate $G$. Then
\begin{equation}\label{cap-6.2}
  \H^c(G)=\bigcap_{T\in G}\H^c(T)=\bigcap_{i=1}^r\H^c(T_i).
\end{equation}
\end{thm}

\subsection{Bergelson-Leibman Theorem}\

\medskip

Given a nilpotent group $G$ of unitary operators on
a Hilbert space $\H$ and a $G$-polynomial $g(n)$ satisfying $g(0) = e_G$, Let $E$ be the
subgroup of $G$ generated by the elements of $g$ and let
$$H^{(l)}(g) =\{u\in \H: P^lu=u \ \text{for all $P\in E$}\}, \ l\in \N.$$
And let
\begin{equation*}
  \H^{\rm rat}(g)=\bigcup_{l\in \N} \H^{(l)}(g).
\end{equation*}

\begin{thm}[Bergelson-Leibman]\cite[Theorem D.]{BL2002}\label{BL}
Let $G$ be a finitely generated nilpotent group of unitary operators on
a Hilbert space $\H$. If $g(n)$ is a $G$-polynomial with $g(0) = e_G$ and  $u\in \H$ is such that
$u\perp \H^{\rm rat}(g)$, then $$\lim_{N\to \infty} \frac{1}{2N+1}\sum_{n=-N}^N g(n)u=0.$$
\end{thm}

\subsection{Some basic notations and facts}\

\medskip

To prove Theorem B, we need more notations and facts.
Recall that a {\em unitary representation} of
$G$ on a Hilbert space $\H$ is a map $G \rightarrow  \B(\H)$ which is a homomorphism into the
group of unitary operators on $\H$.
If $U$ and $V$ are unitary
operators on $\H$ and $\mathcal{K}$, respectively, then there is a unitary operator $U \otimes V$
on $\H \otimes \mathcal{K}$ determined on elementary tensors by $(U\otimes V)(\xi\otimes \zeta)=U\xi \otimes V\zeta$.
The tensor product
$\pi\otimes \rho$ of unitary representations $\pi : G\rightarrow  \B (\H)$ and $\rho : G\rightarrow \B(\mathcal{K})$
is the unitary
representation of $G$ given by $s\mapsto \pi(s)\otimes \rho(s)$.

\medskip

Let $\H $ be a Hilbert space. Its conjugate $\overline{\H}$ is the Hilbert space which is the same
as $\H$ as an additive group but with the scalar multiplication $(c, \xi) \mapsto \bar{c}\xi$ for $c\in \C$
and inner product $\langle \overline{\xi}, \overline{\zeta} \rangle_{\overline{\H}}=\overline{\langle \xi, \zeta\rangle}_{{\H}}$.
If $u\in \H$, we write $\overline{u}$ for the corresponding element in $\overline{\H}$.
If $U$ is a unitary operator on $\H$, then the operator $\overline{U}$
on $\overline{\H}$ which formally coincides with $U$ is also unitary. Given a unitary representation
$\pi: G\rightarrow \B(\H)$, its conjugate $\overline{\pi} : G \rightarrow  \B(\overline{\H})$ is the unitary
representation defined by $s\mapsto \overline{\pi(s)}$.

\medskip

We will need the following well known facts:

\begin{thm}\cite[Lemma 4.18.]{F}\label{Fur}
Let $U$ and $U'$ be unitary operators on Hilbert spaces $\H$ and $\H'$ respectively and let
$w\in \H\otimes \H'$ be an eigenvectors of $U\times U'$: $U\otimes U' w= \lambda w$. Then
$w=\sum_n c_n u_n\otimes u_n'$ where $Uu_n=\lambda_n u_n, U'u_n'=\lambda_n'u_n$ and
$\lambda_n\lambda_n'=\lambda$, and the sequences $\{u_n\}_n, \{u_n'\}_n$ are orthonormal.
\end{thm}

\begin{prop}\label{eigenvectors}
Let $T$ be a unitary operator on $\H$ and $k\in \N$. Then $\H^c(T)$ is the closure of the space spanned by eigenvectors of $T$, and
\begin{equation*}\label{}
  \H^c(T^k)=\H^c(T).
\end{equation*}
\end{prop}

\subsection{Proof of Theorem B}
\

\medskip

Now we begin to prove Theorem B. First we need a result about a decomposition of $\H$.

\begin{prop}\label{zero-density-sequence}
Let $G$ be a finitely generated nilpotent group of unitary operators on a Hilbert space $\H$ and $g$
 be a $G$-polynomial with $g(0) = e_G$. Assume that $E$ is the subgroup of $G$ generated by $g(\Z)$.
Then for each $u\in (\H^c(E))^{\perp}$, one has that
for all $v\in \H$,
\begin{equation}\label{s6}
  \lim_{N\to \infty}\frac{1}{2N+1} \sum_{n=-N}^N |\langle g(n)u,v\rangle|=0.
\end{equation}
\end{prop}

\begin{proof}
Let $\pi: G\rightarrow \B(\H)$ be the identical unitary representation. We consider the tensor product
$\pi\otimes \overline{\pi}$ on $\H\otimes \overline{\H}$. Let $u\in (\H^c(E))^{\perp}$. We will show that
$u\otimes \overline{u} \perp (\H\otimes \overline{\H})^{\rm rat}(g)$. Then by Theorem \ref{BL},
\begin{equation}\label{}
   \lim_{N\to \infty} \frac{1}{2N+1}\sum_{n=-N}^N |\langle g(n)u, v\rangle |^2 = \lim_{N\to \infty}
   \frac{1}{2N+1}\sum_{n=-N}^N \langle \pi\otimes \overline{\pi}(g(n))(u\otimes  \overline{u}), v\otimes \overline{v}\rangle=0 ,
\end{equation}
which is equivalent to (\ref{s6}).

\medskip

Now we show that $u\otimes \overline{u} \perp (\H\otimes \overline{\H})^{\rm rat}(g)$.
Let $w\in (\H\otimes \overline{\H})^{\rm rat}(g)$. Then by definition there is some $l$ such that $w\in
(\H\otimes \overline{\H})^{(l)}(g)$, which means that $P^lw=w$ for all $P\in E$.
For each fixed $P\in E$, by Theorem \ref{Fur}, $w=\sum_n c_n v_n\otimes v_n'$ where $P^lv_n=\lambda_n v_n,
P^lv_n'=\lambda_n'v_n$ and $\lambda_n\lambda_n'=1$, and the sequences $\{v_n\}_n, \{v_n'\}_n$ are orthonormal.
By Proposition \ref{eigenvectors}, $v_n\in \H^c(P^l)=\H^c(P)$ and $v_n'\in \overline{\H}^c(P^l)=\overline{\H}^c(P)$. Then we have
$w\in \H^c(P)\otimes\overline{\H}^c(P)$.

Since $E$ is a subgroup of finitely generated nilpotent group $G$, $E$ is also a finitely generated nilpotent group.
Let $T_1,T_2,\ldots, T_r$ generate $E$.
Then $$w\in  \bigcap_{i=1}^r\H^c(T_i)\otimes\overline{\H}^c(T_i)=\big( \bigcap_{i=1}^r\H^c(T_i)\big)\otimes
\big( \bigcap_{i=1}^r\overline{\H}^c(T_i)\big) =\H^c(E)\otimes \overline{\H}^c(E)$$
where the last equality comes from \eqref{cap-6.2}.

Now since $u\in (\H^c(E))^{\perp}$, one has $\overline{u}\in (\overline{\H}^c(E))^{\perp}$. Hence
$u\otimes \overline{u}\perp \H^c(E)\otimes \overline{\H}^c(E)$. In particular, $u\otimes \overline{u}\perp w$.
Since $w\in (\H\otimes \overline{\H})^{\rm rat}(g)$ is arbitrary, $u\otimes \overline{u}\perp (\H\otimes \overline{\H})^{\rm rat}(g)$.
The proof of Proposition is completed.
\end{proof}

\begin{rem}\label{remark for H^cE invariant}
One can restate Proposition \ref{zero-density-sequence} as follows.
Let $G$ be a finitely generated nilpotent group of unitary operators on a Hilbert space $\H$ and $g$ be a
$G$-polynomial with $g(0) = e_G$. Assume that $E$ is the subgroup of $G$ generated by $g(\Z)$. Then we have the following decomposition
\begin{equation*}
  \H=\H^c(E)\oplus \H^{wm}(E),
\end{equation*}
and $\H^{wm}(E)=\{u\in \H: \lim_{N\to \infty}\frac{1}{2N+1} \sum_{n=-N}^N |\langle g(n)u,v\rangle|=0 \ \text{for all $v\in \H$ } \}$. Note that $\H^c(E)$ and $\H^{wm}(E)$ are $E$-invariant.
\end{rem}

%





\medskip

Using the language of $G$-polynomials, we can restate Theorem B as follows:

\begin{thm}
Let $G$ be a finitely generated nilpotent group of unitary operators on a Hilbert space $\H$ and $g$ be a
$G$-polynomial. Then for all $u,v\in \H$ the sequence $$\{a_n=\langle g(n)u, v \rangle \}_{n\in \Z}$$ is a
sum of a nilsequence and a zero-density-sequence, i.e. an almost nilsequence.
\end{thm}

\begin{proof}

Without loss generality, one may assume that $g(0)=e_G$ (if necessary we replace $g(n)$ and
$u$ by $g(n)(g(0))^{-1}$ and $g(0)u$). Let $E$ be the subgroup of $G$ generated by $g(\Z)$.
Then by \eqref{decomposition}, we have the decompositions
$$u=u_c+u_w, \ v=v_c+v_w,$$
where $u_c,v_c\in \H^c(E), u_w,v_w\in \H^{wm}(E)$. Thus
\begin{align*}
\langle g(n)u,v \rangle&=\langle g(n)u_c, v_c \rangle+ \langle g(n)u_c, v_w \rangle
+ \langle g(n)u_w, v_c \rangle + \langle g(n)u_w, v_w \rangle\\
&=\langle g(n)u_c, v_c \rangle+  \langle g(n)u_w, v_w \rangle
\end{align*}
Since $u_w\in (\H^c(E))^{\perp}$,  the latter one term is zero-density-sequences by Proposition~\ref{zero-density-sequence}.
So it is left to show that $\langle g(n)u_c, v_c \rangle$ is a nilsequence.
It will follows from the following Claim.

\medskip

\noindent {\bf Claim:} \ {\em For all $E$-polynomials $h(n)$ with $h(0)=e_G$ and $u_1, u_2 \in  \H^c(E)$, one has that
$\langle h(n)u_1, u_2 \rangle$ is a nilsequence.
}

\medskip

\noindent {\em Proof of Claim:}\quad By \eqref{polynomial}, since $E$ is a nilpotent group and $h(n)$ is a $E$-polynomial, one can find $U_1,\ldots, U_r\in E$ such that
 $$h(n)=U_1^{p_1(n)}U_2^{p_2(n)}\cdots U_r^{p_r(n)},$$
 where $p_1,\ldots,p_r$ are integral polynomials with $p_1(0) =\ldots = p_r (0) = 0$. We prove it by induction on $r$.

\medskip

First $r=1$. Let $h(n)=U_1^{p_1(n)}$ such that $U_1\in E$ and $p_1$ is an integer polynomial with $p_1(0)=0$. Let $u_1,u_2\in \H^c(E)$. We show that $\langle h(n)u_1, u_2 \rangle$ is a nilsequence.

\medskip

Since $u_1\in \H^c(E)\subseteq \H^c(U_1)$ and $\H^c(U_1)$ is the closure of the linear span of
eigenvectors of $U_1$, for each $\ep>0$ there is $u'=\sum_{j=1}^k c_j f_j$ such that
$\| u_1-u'\|<\ep/{(\|u_2\|+1)}$, where $U_1f_j=e(\theta_j) f_j$ for each $1\le j\le k$.

Note that
\begin{equation*}\label{s1}
\langle U_1^{p_1(n)}u', u_2\rangle  =\sum_{j}c_j \langle U_1^{p_1(n)}f_j, u_2\rangle = \sum_{j}c_j e(p_1(n)\theta_j) \langle f_j, u_2 \rangle
\end{equation*}
is a nilsequence, since $\{e(p_1(n)\theta_j)\}_{n\in \Z}$ is a nilsequence for each $j$.

Now
\begin{equation*}\label{s2}
\begin{split}
   & |\langle U_1^{p_1(n)}u_1, u_2\rangle-\langle U_1^{p_1(n)}u', u_2\rangle|\le \|U_1^{p_1(n)}(u_1-u')\|\ \| u_2\|\\
   &=\|u_1-u'\| \|u_2\|\le \ep ,
\end{split}
\end{equation*}
which means that $\{\langle U_1^{p_1(n)}u_1, u_2\rangle\}_{n\in \Z}$ is the limit of nilsequences in $l^\infty(\Z)$.
Hence $\{\langle U_1^{p_1(n)}u_1, u_2\rangle\}_{n\in \Z}$ is a nilsequence itself.

\medskip

Now assume the claim holds for $r$, and we show the case for $r+1$. Let
$$h(n) = U_1^{p_1(n)}\ldots U_{r+1}^{p_{r+1} (n)}, $$
where $U_1,\ldots,U_{r+1}\in E$ and $p_1,\ldots,p_{r+1}$ are integral polynomials with $p_1(0) =\ldots = p_r (0) =p_{r+1}(0)= 0$. We
will show that $\{\langle h(n)u_1, u_2\rangle\}_{n\in \Z}$ is a nilsequence.

\medskip

Notice that $\H^c(E)$ is invariant under $U_{r+1}$. Since $\H^c(U_{r+1}|_{\H^c(E)})$ is the closure of the linear span of eigenvectors of $U_{r+1}$ on $\H^c(E)$,
for each $\ep>0$ there is some $u'=\sum_{j=1}^k c_j f_j$ such that $\| u_1-u'\|<\ep/{(\|u_2\|+1)}$,
where $f_j\in \H^c(E)$ and $U_{r+1}f_j=e(\theta_j) f_j$ for each $1\le j\le k$.

We have that
\begin{equation*}\label{s3}
  \begin{split}
\langle h(n)u', u_2\rangle &= \langle U_1^{p_1(n)}\ldots U_r^{p_r(n)} U_{r+1}^{p_{r+1}(n)}u', u_2\rangle \\ & =\sum_{j}c_j \langle U_1^{p_1(n)}\ldots U_r^{p_r(n)} U_{r+1}^{p_{r+1}(n)}f_j, u_2\rangle\\ & = \sum_{j}c_j e(p_{r+1}(n)\theta_j) \langle U_1^{p_1(n)}\ldots U_r^{p_r(n)} f_j, u_2\rangle.
 \end{split}
\end{equation*}
Since $h'(n)=U_1^{p_1(n)}\ldots U_r^{p_r(n)}$ is still a $E$-polynomial with $p_1(0) =\ldots = p_r (0) = 0$ and $f_j, u_2\in \H^c(E)$,
by inductive assumption $\{\langle U_1^{p_1(n)}\ldots U_r^{p_r(n)} f_j, u_2\rangle\}_{n\in \Z}$ is
a nilsequence. Combining this with the fact $\{e(p_1(n)\theta_j)\}_{n\in \Z}$ being a nilsequence for each $j$, one has that
$\{\langle h(n)u', u_2\rangle\}_{n\in \Z}$ is a nilsequence.

Now
\begin{equation*}\label{s4}
\begin{split}
   & |\langle h(n)u_1, u_2\rangle-\langle h(n)u', u_2\rangle|\le \|h(n)(u_1-u')\|\ \| u_2\|\\
   &=\|u_1-u'\| \|u_2\|\le \ep .
\end{split}
\end{equation*}
That means $\{\langle h(n)u_1, u_2\rangle\}_{n\in \Z}$ is the limit of nilsequences in $l^\infty(\Z)$.
Hence $\{\langle h(n)u_1, u_2\rangle\} _{n\in \Z}$ is a nilsequence itself.
By induction the proof is completed.
\end{proof}




\bigskip


\end{document}